\documentclass[12pt]{amsart}
\usepackage{graphicx,colordvi}
\usepackage{amsmath,amsfonts,amssymb}

\newtheorem{theorem}{Theorem}[section]
\newtheorem{lemma}[theorem]{Lemma}
\newtheorem{proposition}[theorem]{Proposition}

\theoremstyle{definition}
\newtheorem{definition}[theorem]{Definition}
\newtheorem{example}[theorem]{Example}

\theoremstyle{remark}
\newtheorem{remark}[theorem]{Remark}

\numberwithin{equation}{section}

\newcounter{FNC}[page]
\def\fauxfootnote#1{{\addtocounter{FNC}{2}$^\fnsymbol{FNC}$%
     \let\thefootnote\relax\footnotetext{$^\fnsymbol{FNC}$#1}}}

\newcommand{\N}{{\mathbb N}}
\newcommand{\Q}{{\mathbb Q}}
\newcommand{\Z}{{\mathbb Z}}

\newcommand{\frakS}{{\mathfrak S}}

\newcommand{\Des}{\ensuremath\mathrm{Des}}
\newcommand{\ep}{\ensuremath\mathrm{end}}
\newcommand{\up}{\ensuremath\mathrm{up}}
\newcommand{\dw}{\ensuremath\mathrm{dw}}

\newcommand{\calA}{{\mathcal A}}
\newcommand{\calC}{{\mathcal C}}
\newcommand{\calG}{{\mathcal G}}
\newcommand{\precdot}{{\prec\!\!\!\cdot\,}}

\makeatletter
\newcommand\xleftrightarrow[2][]{%
  \ext@arrow 9999{\longleftrightarrowfill@}{#1}{#2}}
\newcommand\longleftrightarrowfill@{%
  \arrowfill@\leftarrow\relbar\rightarrow}
\makeatother

\newcommand{\FThT}[9]{\begin{picture}(52,39.5)(-2.9,-2.5)
  \put(-3.4,-3){\includegraphics{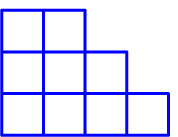}}
  \put( 0,23.3){{\small#8}}\put(11.75,23.3){\small#9}
  \put( 0,11.5){{\small#5}}\put(11.75,11.5){\small#6}
  \put(23.5,11.5){\small#7}
  \put( 0,-0.3){\small#1}\put(11.75,-0.3){\small#2}
  \put(23.5,-0.3){\small#3}\put(35.25,-0.3){\small#4}
 \end{picture}}
\newcommand{\ThT}[5]{\begin{picture}(39.5,27.5)(-2.9,-2.5)
  \put(-3.4,-3){\includegraphics{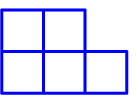}}
  \put( 0,11.5){{\small#4}}\put(11.75,11.5){\small#5}
  \put( 0,-0.3){\small#1}\put(11.75,-0.3){\small#2}
  \put(23.5,-0.3){\small#3}
 \end{picture}}
\newcommand{\ThII}[5]{\begin{picture}(39.5,39.5)(-2.9,-2.5)
  \put(-3.4,-3){\includegraphics{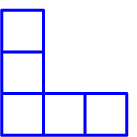}}
  \put( 0,23.3){{\small#5}}
  \put( 0,11.5){{\small#4}}
  \put( 0,-0.3){\small#1}\put(11.75,-0.3){\small#2}
  \put(23.5,-0.3){\small#3}
 \end{picture}}
\newcommand{\dt}{\begin{picture}(0.01,1)\put(-3,-4){\Red{\huge .}}\end{picture}}
\def\Color#1#2{#2}
\def\MyGreen#1{\Color{0 0.6 0}{#1}}
\def\MyMag#1{\Color{.8 0 .8}{#1}}
\def\MyCyan#1{\Color{0 0.2 1}{#1}}
\def\MyMar#1{\Color{0.5 0 0}{#1}}
\newcommand{\defcolor}[1]{\Cerulean{#1}}
\newcommand{\demph}[1]{\defcolor{{\sl #1}}}

\begin{document}


\title[Nonnegativity of Schubert coefficients]{A combinatorial proof
 that Schubert vs.\ Schur coefficients are nonnegative} 

\author[Assaf]{Sami Assaf}
\address{Department of Mathematics, University of Southern California,
  Los Angeles, CA 90089, USA}
\email{shassaf@usc.edu}

\author[Bergeron]{Nantel Bergeron}
\address{Department of Mathematics and Statistics, York University,
  Toronto, Ontario M3J 1P3, Canada}
\email{bergeron@mathstat.yorku.ca}
\thanks{Bergeron was supported in part by CRC and  NSERC}

\author[Sottile]{Frank Sottile}
\address{Department of Mathematics, Texas A\&M University, College
  Station, TX 77843, USA} 
\email{sottile@math.tamu.edu}
\thanks{Sottile was supported in part by  DMS-1001615}

\subjclass[2010]{05E05, 14M15}

\keywords{Grassmannian-Bruhat order, flag manifold,  Grassmannian,
  Littlewood--Richardson rule, Schubert variety, quasisymmetric
  functions, Schur functions, positivity}

\begin{abstract}
  We give a combinatorial proof that the product of a Schubert polynomial 
  by a Schur polynomial is a nonnegative sum of Schubert polynomials.
  Our proof uses Assaf's theory of dual equivalence to show that 
  a quasisymmetric function of Bergeron and Sottile is Schur-positive. 
  By a geometric comparison theorem of Buch and Mihalcea, this implies the nonnegativity of
  Gromov-Witten invariants of the Grassmannian.
\end{abstract}

\maketitle

\dedicatory{Dedicated to the memory of Alain Lascoux}

\section*{Introduction}
\label{sec:introduction}

A Littlewood-Richardson coefficient is the multiplicity of an irreducible representation
of the general linear group in a tensor product of two irreducible representations, and is
thus a nonnegative integer. 
Littlewood and Richardson conjectured a formula for these coefficients
in 1934~\cite{LR1934}, which was proven in the 1970's by Thomas~\cite{Thomas} and
Sch\"utzenberger~\cite{Sc_jdt}.
Since Littlewood-Richardson coefficients may be defined combinatorially as the
coefficients of Schur functions in the expansion of a product of two Schur 
functions, these proofs of the Littlewood-Richardson rule 
furnish combinatorial proofs of the nonnegativity of Schur structure constants.

The Littlewood-Richardson coefficients are also the structure constants for expressing 
products in the cohomology of a Grassmannian in terms of its basis of Schubert
classes. 
Independent of the Littlewood-Richardson rule, these Schubert structure constants are known
to be nonnegative integers through geometric arguments.
The integral cohomology ring of any flag manifold has a Schubert basis 
and again geometry implies that the corresponding Schubert structure constants are
nonnegative integers.
These cohomology rings and their Schubert bases have combinatorial models, and it
remains an open problem to give a combinatorial proof that the
Schubert structure constants are nonnegative.

We give such a combinatorial proof of nonnegativity for a class of Schubert structure
constants in the classical flag manifold.
These are the constants that occur in the product of an arbitrary Schubert
class by one pulled back from a Grassmannian projection.
They are defined combinatorially as the product of a Schubert polynomial~\cite{LS82} by a
Schur symmetric polynomial; we call them \demph{Schubert vs.\ Schur coefficients}.
As the Schubert polynomials form a basis for the ring of polynomials in $z_1,z_2,\dotsc$,
these coefficients determine its structure as a module over polynomials symmetric in
$z_1,\dotsc,z_k$, for any $k$.  
These coefficients were studied by Bergeron and Sottile~\cite{BS_Duke,Monoide,BS_Skew}, who
defined a quasisymmetric generating function associated to intervals in the
Grassmannian-Bruhat order (a partial order on the symmetric group), showed that this
quasisymmetric function is symmetric, and that the  
Schubert vs.\ Schur coefficients are the coefficients of Schur functions in this symmetric 
function.
This work relied upon the Pieri formula~\cite{PieriSchubert}, which has since been given
combinatorial proofs~\cite{K_Pieri,Postnikov}.
Consequently, a combinatorial proof that these quasisymmetric functions are
Schur-positive gives a combinatorial proof of nonnegativity of the Schubert vs.\ Schur
coefficients.

Another important question of combinatorial positivity concerned the Macdonald
polynomials~\cite{Mac95}.
These are symmetric functions with coefficients in the field $\Q(q,t)$ of rational
functions in variables $q,t$,
which were conjectured by Macdonald to have coefficients in $\N[q,t]$ (polynomials in $q,t$
with nonnegative integer coefficients) when expanded in the Schur basis.
Five nearly simultaneous proofs in the mid 1990's showed that these coefficients were in
$\Z[q,t]$~\cite{GR98,GT96,KN98,Kn97,Si96}. 
While positivity was proven by Haiman~\cite{Ha01} using algebraic geometry and
representation theory, a combinatorial proof took longer.
A breakthrough came when Macdonald polynomials were given a combinatorial definition as a
quasisymmetric function with coefficients in $\N[q,t]$~\cite{HHL}.
That work also showed them to have a positive expansion in terms of LLT
polynomials~\cite{LLT}. 
The Macdonald positivity conjecture was finally given a combinatorial proof by
Assaf~\cite{A_JAMS}, who introduced a new technique---dual equivalence graphs---for
proving Schur-positivity of quasisymmetric functions and applied it to the LLT
polynomials. 
Dual equivalence not only shows Schur-positivity, but it also gives a  (admittedly complicated)
combinatorial formula for the Schur coefficients in the quasisymmetric function.

When the  quasisymmetric function comes from descents in a collection of words
the theory of dual equivalence may be recast in terms of a family of involutions on the words~\cite{A_TAMS}.
The critical condition for a dual equivalence only needs to be verified
locally on all subwords of length up to five and some of length six.
The symmetric function of Bergeron and Sottile is the quasisymmetric function associated
to descents on a collection of words that are themselves saturated chains in intervals of
the Grassmannian-Bruhat order.
The results of~\cite{Monoide} lead to a unique family of involutions which immediately
satisfy most properties of a dual equivalence for these chains.
All that remains is to verify the local conditions on chains of length up to six. 
There are only finitely many chains of a given length, up
to an equivalence, and thus we may verify this local condition on a computer.
This shows that the symmetric function of Bergeron and Sottile is Schur-positive and
gives a combinatorial proof that the Schubert vs.\ Schur constants are
nonnegative.  

Identifying the fundamental class of a Schubert variety with a
Schur function identifies the homology of the Grassmannian with a linear subspace in the
algebra of symmetric functions.
Under this identification, each symmetric function of Bergeron and Sottile is the
fundamental cycle of the image in the Grassmannian of a Richardson variety (intersection
of two Schubert varieties) under the projection from the flag variety.
These Richardson images are also known intrinsically as positroid varieties~\cite{KLS}.
Buch, Chaput, Mihalcea, and Perrin~\cite{BCMP} showed that each structure constant in the
quantum cohomology of the Grassmannian~\cite{Bertram} (quantum Littlewood-Richardson
numbers) naturally arises when 
expressing a certain projected Richardson class in the Schubert basis of the homology of the
Grassmannian, and is thus naturally a Schubert vs.\ Schur coefficient.

Our combinatorial proof of positivity of Schubert vs.\ Schur coefficients does 
not yield a combinatorial proof of positivity of the quantum
Littlewood-Richardson coefficients (the comparison theorem in~\cite{BCMP} is geometric),
but it does give a (complicated) combinatorial formula for those coefficients.
Recent work of Buch, Kresch, Purbhoo, and Tamvakis~\cite{BKPT} giving 
a combinatorial formula for the structure constants of two-step flag manifolds also gives
a combinatorial formula for the quantum Littlewood-Richardson
coefficients as each is equal to some Schubert structure constant on a two-step flag
manifold~\cite{BKT03}. 
This however does not give a combinatorial proof of nonnegativity, as the comparison
theorem in~\cite{BKT03} is geometric.

This paper is organized as follows.
Section~\ref{S:one} develops background material of quasisymmetric functions, Assaf's dual 
equivalence, and the work of Bergeron and Sottile.
We also state our main theorem (Theorem~\ref{Th:Main}) in Section~\ref{S:one}. 
In Section~\ref{S:two} we show that there is at most one dual equivalence on the set
of chains in intervals in the Grassmannian-Bruhat order satisfying some natural properties
(Lemma~\ref{L:locality}), and then we prove that this is a dual equivalence
(Theorem~\ref{Th:DualEquivalance}), which implies our main result.

\section{Quasisymmetric functions, dual equivalence, and Schubert coefficients}\label{S:one}

We collect here background on quasisymmetric functions, recall the salient parts of
Assaf's theory of dual equivalence, define the symmetric function of Bergeron and
Sottile, and connect it to the Schubert vs.\ Schur coefficients.

\subsection{Symmetric functions and tableaux}

The algebra \defcolor{$\Lambda$} of symmetric functions is freely generated by the
complete symmetric 
functions $h_1,h_2,\dotsc$, where $h_m$ is the formal sum of all monomials of degree $m$ in the
countably many variables $x_1,x_2,\dotsc$.
Thus $\Lambda$ has a basis of monomials in the $h_m$.
More interesting is its basis of Schur functions $s_\lambda$, which are indexed by
partitions $\lambda$ and are generating functions for Young tableaux.

A \demph{partition $\lambda$} is a weakly decreasing finite sequence of integers
$\lambda\colon\lambda_1\geq\dotsb\geq\lambda_k\geq0$.
Write $|\lambda|$ for the sum $\lambda_1+\dotsb+\lambda_k$.
We represent a partition $\lambda$ by its \demph{Young diagram}, which is the
left-justified array of boxes with $\lambda_i$ boxes at height $i$. 
Thus
\[
   (3,2)\ \longleftrightarrow\ 
     \raisebox{-6pt}{\includegraphics{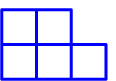}}
   \qquad\mbox{and}\qquad
   (5,4,2,1)\ \longleftrightarrow\ 
     \raisebox{-16pt}{\includegraphics{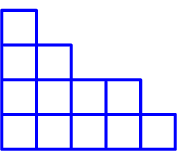}}\,.
\]
A \demph{Young tableau} is a filling of $\lambda$ with positive integers that weakly
increase across each row and strictly increase up each column.
It is \demph{standard} if the integers are $1,2,\dotsc,|\lambda|$.
Here are four Young tableaux of shape $(4,3,2)$.
Only the last is standard.
 \begin{equation}\label{Eq:432Tableaux}
   \raisebox{-15pt}{\FThT{1}{1}{2}{2}{2}{3}{3}{4}{4}\qquad
   \FThT{1}{1}{1}{2}{2}{3}{4}{4}{5}\qquad
   \FThT{1}{1}{4}{4}{2}{5}{8}{7}{8}\qquad
   \FThT{1}{2}{4}{5}{3}{6}{8}{7}{9}}
 \end{equation}

To a tableau $T$ of shape $\lambda$, we associate a monomial
\[
    \defcolor{x^T}\ :=\ \prod_{i\in T} x_i\,,
\]
the product is over all entries $i$ of $T$.
The tableaux in~\eqref{Eq:432Tableaux} give the monomials
\[
   x_1^2x_2^3x_3^2x_4^2\,,\ 
   x_1^3x_2^2x_3x_4^2x_5\,,\ 
   x_1^2x_2x_4^2x_5x_7x_8^2\,,\ 
   \ \mbox{and}\ \ 
   x_1x_2x_3x_4x_5x_6x_7x_8x_9\,,
\]
respectively.
The Schur function $s_\lambda$ is the generating function for tableaux of shape
$\lambda$,
\[
   s_\lambda\ =\ \sum_T x^T\,,
\]
the formal sum over all tableaux $T$ of shape $\lambda$.

\subsection{Quasisymmetric functions}

Gessel's quasisymmetric functions arise as generating functions for enumerative combinatorial
invariants~\cite{Gessel}.
A formal power series $F(x)$ in countably many variables $x_1,x_2,\dotsc$ 
having bounded degree is \demph{quasisymmetric} if for any list of positive integers
$(\alpha_1,\dotsc,\alpha_n)$ and increasing sequence of positive integers
$i_1<\dotsb<i_n$, the coefficient of the monomial
 \[
   x_{i_1}^{\alpha_1} x_{i_2}^{\alpha_2} \dotsb  x_{i_n}^{\alpha_n} 
 \]
in $F(x)$ does not depend upon the choice of $i_1<\dotsb<i_n$.
For example, 
\[
   2x_1^2x_2+2x_1^2x_3+2x_2^2x_3 + \dotsb\quad
   -x_1x_2x_3-x_1x_2x_4-x_1x_3x_4-x_2x_3x_4-\dotsb
\]
is quasisymmetric.

A fundamental quasisymmetric function \defcolor{$Q_D(x)$} is given by a positive integer
$n$ and a subset $D$ of $\defcolor{[n{-}1]}:=\{1,\dotsc,n{-}1\}$ and defined to be 
\[
   Q_D(x)\ :=\ \sum_{\substack{i_1 \leq \cdots \leq i_n \\ j \in D
      \Rightarrow i_j < i_{j+1}}} x_{i_1} \cdots x_{i_n}\, .
\]
The degree $n$ is implicit in our notation.
These form a basis of quasisymmetric functions. 

Given a set $\calC$ of combinatorial objects and a map $\Des$ 
from $\calC$ to the subsets of $[n{-}1]$, we define the
quasisymmetric generating function of $(\calC,\Des)$,
\[
   \defcolor{K_{(\calC,\Des)}}\ =\ \defcolor{K_{\calC}}\ 
    :=\ \sum_{c\in\calC} Q_{\Des(c)}(x)\,.
\]
A source for this is when $\calC$ is a set of (saturated) chains in a
labeled poset of rank $n$.
Here, a labeled poset is a finite ranked poset $P$ together with an
integer label on each cover of $P$.
Given a chain $c$ in $P$, the sequence of labels of its covers is a word $w$ of
length $n$, and we let $\Des(c):=\{i\in[n{-}1]\mid w_i>w_{i+1}\}$, the descent
set of $w$.

A map $f\colon P\to Q$ of finite labeled posets is a \demph{label-equivalence} if
$f$ is an isomorphism of posets and if the labels in $P$ occur in the same relative
order as the labels in $Q$. 
That is, if for any two covers $u\lessdot v$ and 
$x\lessdot y$ in $P$ with labels $a$ and $b$, respectively, if $\alpha$ and $\beta$ are
the labels of the corresponding covers $f(u)\lessdot f(v)$ and 
$f(x)\lessdot f(y)$ in $Q$, then $a\leq b$ if and only if $\alpha\leq\beta$. 
Label-equivalent posets have identical quasisymmetric functions.

\begin{example}\label{Ex:YT}
 Partitions are partially ordered by containment of their Young diagrams and the
 resulting poset is Young's lattice.
 A cover is given by adding a box in row $i$ and column $j$ and is labeled 
 with $j{-}i$ to obtain a labeled poset.
 Standard tableaux of shape $\lambda$ correspond to saturated chains from $\emptyset$ to
 $\lambda$:  
 the box containing the integer $n$ is the box corresponding to the $n$th cover in the
 chain.
 Figure~\ref{F:YL} shows part of Young's lattice.
\begin{figure}[htb]
  \begin{picture}(169,208)(0,7)
   \put(0,10){\includegraphics{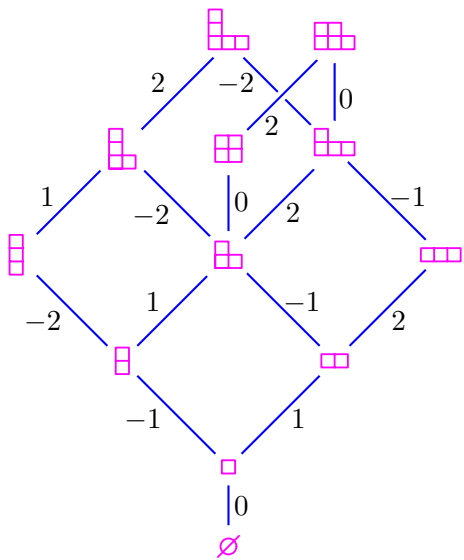}}

   \put( 54,183){\small$2$} \put(79,183){\small$-2$}
   \put( 97,167){\small$2$} \put(125,177){\small$0$}

   \put( 12,140){\small$1$} \put(47,133){\small$-2$} \put(85.5,138){\small$0$}
   \put(105,134){\small$2$} \put(144,140){\small$-1$}

   \put( 6,93){\small$-2$}  \put(52,100){\small$1$}
   \put(104,100){\small$-1$} \put(145,93){\small$2$}

   \put(44,56){\small$-1$} \put(107,56){\small$1$}

   \put(85.5,23){\small$0$}

  \end{picture} 
 \caption{Part of Young's Lattice}\label{F:YL}
\end{figure}

 We display the five Young tableaux of shape $(3,2)$.
 Below each, we give the sequence of labels in the corresponding chain, writing $\bar{a}$ for
 $-a$ and placing a dot at each descent.
 \begin{equation}\label{Eq:Tab32}
  \begin{tabular}{ccccccccc}
    \ThT{1}{2}{5}{3}{4}&\quad& 
    \ThT{1}{3}{5}{2}{4}&\quad&  
    \ThT{1}{3}{4}{2}{5}&\quad&
    \ThT{1}{2}{4}{3}{5}&\quad& 
    \ThT{1}{2}{3}{4}{5}\\ 
     {0}{1}\dt{\={1}}{0}{2}&&
     {0}\dt{\={1}}{1}\dt{0}{2}&&
     {0}\dt{\={1}}{1}{2}\dt{0}&&
     {0}{1}\dt{\={1}}{2}\dt{0}&&
     {0}{1}{2}\dt{\={1}}{0}
  \end{tabular}
 \end{equation}

 We display the six Young tableaux of shape $(3,1,1)$, together with the 
 sequence of labels in the corresponding chain and descent sets.
 \begin{equation}\label{Eq:Tab311}
  \begin{tabular}{ccccccccccc}
    \ThII{1}{2}{3}{4}{5}&\quad& 
    \ThII{1}{2}{4}{3}{5}&\quad&
    \ThII{1}{2}{5}{3}{4}&\quad& 
    \ThII{1}{3}{4}{2}{5}&\quad& 
    \ThII{1}{3}{5}{2}{4}&\quad& 
    \ThII{1}{4}{5}{2}{3}\\ 
     {0}{1}{2}\dt{\={1}}\dt{\={2}}&&
     {0}{1}\dt{\={1}}{2}\dt{\={2}}&&
     {0}{1}\dt{\={1}}\dt{\={2}}{2}&&
     {0}\dt{\={1}}{1}{2}\dt{\={2}}&&
     {0}\dt{\={1}}{1}\dt{\={2}}{2}&&
     {0}\dt{\={1}}\dt{\={2}}{1}{2}
  \end{tabular}
 \end{equation}
\end{example}

A standard young tableau $T$ has a \demph{descent at $i$} when $i{+}1$ is above $i$ in
$T$, equivalently when $i{+}1$ is weakly left of $i$.
(These are the descents we observe in the previous two examples.)
The quasisymmetric function associated to a partition $\lambda$ is 
\[
   \defcolor{K_\lambda}\ :=\ \sum_{T} Q_{\Des(T)}\,,
\]
the sum over all standard Young tableaux $T$ where $\Des(T)$ is the descent set of the
chain in Young's lattice corresponding to $T$.
Gessel~\cite{Gessel} showed that $K_\lambda=s_\lambda$, the Schur function associated to
$\lambda$.
From~\eqref{Eq:Tab32} and~\eqref{Eq:Tab311}, we have
 \begin{eqnarray*}
   s_{(3,2)}&=& 
   Q_{\{3\}}\ +\ 
   Q_{\{2,4\}}\ +\ 
   Q_{\{1,4\}}\ +\ 
   Q_{\{1,3\}}\ +\ 
   Q_{\{2\}}\hspace{50pt}\mbox{and}\\
   s_{(3,1,1)}&=&
   Q_{\{3,4\}}\ +\ 
   Q_{\{2,4\}}\ +\ 
   Q_{\{2,3\}}\ +\ 
   Q_{\{1,4\}}\ +\ 
   Q_{\{1,3\}}\ +\ 
   Q_{\{1,2\}}\,.
 \end{eqnarray*}

\begin{example}\label{Ex:143256}
 Consider the interval $[e,(145326)]_\preceq$ in the Grassmannian-Bruhat
 order~\cite{Monoide}. 
 An edge is a  transposition $t_{ab}=(a,b)$ with $a<b$, which is given the label is $b$.
 \begin{figure}[htb]
  \begin{picture}(253,225)(-125,0)
   \put(-100,0){\includegraphics{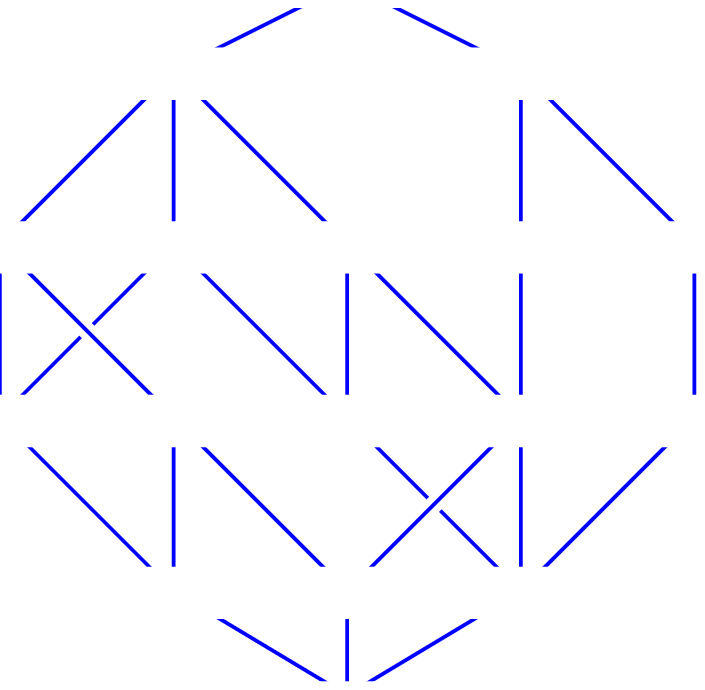}}

   \put(-23,212){$(145326)$}

   \put(-30,205){\smaller$4$}\put(30,205){\smaller$4$}

   \put(-77,185){$(1326)(45)$}   \put(29,185){$(26453)$}

   \put(-87,163){\smaller$3$}\put(-57,163){\smaller$5$}\put(-21,163){\smaller$3$}
                   \put(42,163){\smaller$6$}\put(80,163){\smaller$5$}

   \put(-123,136){$(126)(45)$} \put(-66,136){$(1326)$} \put(-23,136){$(263)(45)$}
         \put(34,136){$(2453)$}   \put(84,136){$(2653)$}

   \put(-108,113){\smaller$5$}\put(-92,119){\smaller$2$}\put(-65,118){\smaller$3$}
    \put(-21,113){\smaller$3$}\put(2,113){\smaller$5$}\put(29,113){\smaller$6$}
                   \put(52,113){\smaller$4$}\put(102,113){\smaller$6$}

   \put(-115,86){$(126)$} \put(-69,86){$(26)(45)$} \put(-13,86){$(263)$}
         \put(29,86){$(23)(45)$}   \put(87,86){$(253)$}

   \put(-85,61){\smaller$2$}\put(-58,67){\smaller$5$}\put(-21,63){\smaller$6$}
                   \put(10,66){\smaller$6$}\put(33,65){\smaller$3$}
                   \put(52,64){\smaller$5$}\put(81,62){\smaller$5$}

   \put(-60,36){$(26)$}\put(-10,36){$(45)$}\put(40,36){$(23)$}

   \put(-33,17){\smaller$6$}\put(-7,23){\smaller$5$}\put(27,17){\smaller$3$}

   \put(-3,7){$e$}

  \end{picture}
 \caption{A labeled poset}\label{F:145326}
 \end{figure}
This poset has eleven chains.
Here is the rightmost
\[
   e\;\xrightarrow{\;3\;}\;
   (23)\;\xrightarrow{\;5\;}\;
   (253)\;\xrightarrow{\;6\;}\;
   (2653)\;\xrightarrow{\;5\;}\;
   (26453)\;\xrightarrow{\;4\;}\;
   (145326)\,.
\]
Its sequence of labels is $35654$, which has descent set $\{3,4\}$.
We list the labels of the chains in this poset, placing a dot under each descent.
 \begin{equation}\label{Eq:ElevenChainLabels}
  \begin{array}{c}
    356\dt34\,,\ 
    36\dt35\dt4\,,\ 
    6\dt235\dt4\,,\ 
    6\dt25\dt34\,,\ 
    56\dt234\,,
   \\
    356\dt5\dt4\,,\ 
    35\dt46\dt4\,,\ 
    36\dt5\dt34\,,\ 
    5\dt346\dt4\,,\ 
    5\dt36\dt34\,,\ 
    6\dt5\dt234\,.\rule{0pt}{13pt}
  \end{array}
 \end{equation}
Thus the quasisymmetric function of this interval is
\[
  Q_{\{3,4\}}\ +\ 
  2Q_{\{2,4\}}\ +\ 
  Q_{\{2,3\}}\ +\ 
  2Q_{\{1,4\}}\ +\ 
  Q_{\{3\}}\ +\ 
  2Q_{\{1,3\}}\ +\ 
  Q_{\{1,2\}}\ +\ 
  Q_{\{2\}}\,,
\]
which is $s_{(3,2)}+s_{(3,1,1)}$.
\hfill\qed
\end{example}

\subsection{Dual Equivalence}

Assaf's theory of dual equivalence~\cite{A_TAMS,A_JAMS} is a general framework for proving
Schur-positivity of quasisymmetric generating functions. 

\begin{definition}\label{D:strong_dual_equivalence}
Let $\calC$ be a finite set, $n$ a positive integer, and $\Des$ a map from $\calC$ to the
subsets of $[n{-}1]$ where, for $c \in \calC$, if $i\in\Des(c)$, we say that $c$ has a
\demph{descent at $i$}.
A \demph{strong dual equivalence} for $\calC$ (or for $(\calC,\Des)$) is a collection
$\{\varphi_i\mid i=2,\dotsc,n{-}1\}$ of involutions on $\calC$ which satisfy the following
conditions.
\begin{enumerate}
 \item[(i)] The fixed points 
   of the involution $\varphi_i$
   are those $c \in \calC$ which either have a descent at both $i{-}1$ and $i$ or do not have a
   descent at either $i{-}1$ or at $i$. 

 \item[(ii)] For elements $c\in\calC$ with $c\neq\varphi_i(c)$, so that $c$ has exactly
   one descent in $\{i{-}1,i\}$, 
  \begin{enumerate}
   \item $c$ and $\varphi_i(c)$ have the same descents, except possibly in 
     $\{i{-}2,i{-}1,i,i{+}1\}$.

   \item For each $j\in\{i{-}1,i\}$ exactly one of $c$ and $\varphi_i(c)$ has a descent at
     $j$. 
 
   \item If exactly one of $c$ and $\varphi_i(c)$ has a descent at $i{-}2$, then 
        $c\neq \varphi_{i-1}(c)$.

   \item If exactly one of $c$ and $\varphi_i(c)$ has a descent at $i{+}1$, then
           $c\neq \varphi_{i+1}(c)$.
  \end{enumerate}
 
 \item[(iii)] If $|i-j|\geq 3$, then $\varphi_i$ and $\varphi_j$ commute.

 \item[(iv)] 
      For any $i<j\leq i{+}3$, if $b=\varphi_{i_\ell}\circ\dotsb\circ\varphi_{i_1}(c)$ for
      indices $i_1,\dotsc,i_\ell$ in the interval $[i,j]$ with $\ell>0$, then there exist
      indices $j_1,\dotsc,j_m$ in the interval $[i,j]$ with $m>0$ and at most one $j_k=j$
      such that $b=\varphi_{j_m}\circ\dotsb\circ\varphi_{j_1}(c)$.

\end{enumerate}
\end{definition}

When $\calC$ is the set of standard Young tableaux $T$ with shape $\lambda$ and $\Des(T)$
the descent set defined after Example~\ref{Ex:YT}, Haiman's notion
of dual equivalence~\cite{Ha92} gives a strong dual equivalence on $\calC$.
Let $T$ be a standard tableau with entries $1,\dotsc,n$.
For each $i=2,\dotsc,n{-}1$ define $\varphi_i(T)$ by the relative positions of
$i{-}1$, $i$, and $i{+}1$ in $T$.
If they are in order left-to-right, so that $T$ has no descents in position $\{i{-}1,i\}$,
we set $\varphi_i(T)=T$.
If $i{+}1$ is in a row above $i$, which is  in a row above $i{-}1$, then $T$ has a descent at
both $i{-}1$ and $i$, and we also set  $\varphi_i(T)=T$.
Otherwise, either $i$ is above both $i{-}1$ and $i{+}1$ or else it is below or to the
right of both. 
In either case, let  $\varphi_i(T)$ be the tableau obtained from $T$ by switching $i$ with
whichever of $i{-}1$ or $i{+}1$ that is further away from it in $T$.

We illustrate this for the partition $(3,2)$, displaying the descent set below each
tableau. 
 \begin{equation}\label{Eq:32_DE}
   \raisebox{-18pt}{\begin{picture}(335,36)(0,-12)
    \put(  0,0){\ThT{1}{2}{5}{3}{4}}
     \put( 8.6,-12){$\{2\}$}
     \put(40,12){$\xleftrightarrow{\;\ \varphi_3\ \;}$}
     \put(40, 6){$\xleftrightarrow[\;\ \varphi_2\ \;]{}$}
    \put( 75,0){\ThT{1}{3}{5}{2}{4}}
     \put( 78,-12){$\{1,3\}$}
     \put(115,9){$\xleftrightarrow{\;\ \varphi_4\ \;}$}
    \put(150,0){\ThT{1}{3}{4}{2}{5}}
     \put(153,-12){$\{1,4\}$}
     \put(190,9){$\xleftrightarrow{\;\ \varphi_2\ \;}$}
    \put(225,0){\ThT{1}{2}{4}{3}{5}}
     \put(228,-12){$\{2,4\}$}
     \put(265,12){$\xleftrightarrow{\;\ \varphi_4\ \;}$}
     \put(265, 6){$\xleftrightarrow[\;\ \varphi_3\ \;]{}$}
    \put(300,0){\ThT{1}{2}{3}{4}{5}}
     \put(308.6,-12){$\{3\}$}
  \end{picture}}
 \end{equation}

\demph{Strong dual equivalence} is the relation \defcolor{$\sim$} on $\calC$ generated by 
$c\sim\varphi_i(c)$, for $c\in\calC$ and $i=2,\dotsc,n{-}1$, when the $\varphi_i$
satisfy the conditions of Definition~\ref{D:strong_dual_equivalence}.
A main result of~\cite{A_JAMS} is that if $\calA$ is a strong dual equivalence class, then
there is a partition $\lambda$ such that 
 \begin{equation}\label{Eq:DE_Schur_Coeff}
   K_\calA\ =\ \sum_{a\in\calA} Q_{\Des(a)}\ =\ s_\lambda\,.
 \end{equation}
This leads to a combinatorial formula (which we do not state) for the Schur
coefficients $c^\lambda_{\calC}$ defined by the identity
\[
   K_{(\calC,\Des)}\ =\ \sum_\lambda c^\lambda_{\calC} s_\lambda\,,
\]
where $(\calC,\Des)$ as above admits a strong dual equivalence structure.

Condition (iv) for a strong dual equivalence is hard to satisfy and to check.
Assaf introduced a weaker notion which implies that the quasisymmetric generating function of
each equivalence class is Schur-positive, but not necessarily equal to a single Schur
function.

Suppose that we have a set $\calC$ and a notion of descent $\Des$ for elements of $\calC$ and
involutions $\varphi_i$ for $i=2,\dotsc,n{-}1$ as above.
Given $2\leq i<j\leq n{-}1$, we may restrict $\Des$ and the equivalence relation $\sim$ to
the interval $[i,j]$ as follows.
For $c\in\calC$, define $\defcolor{\Des_{(i,j)}}(c)\subset\{1,\dotsc,j{-}i{+}2\}$ by first
intersecting $\Des(c)$ with the interval $[i{-}1,j]$ and then subtracting $i{-}2$
from each element.
Similarly, \defcolor{$\sim_{(i,j)}$} is the coarsest equivalence relation on $\calC$ in which
$c\sim_{(i,j)}\varphi_k(c)$ where $i\leq k\leq j$.
Write \defcolor{$[c]_{(i,j)}$} for the $\sim_{(i,j)}$-equivalence class containing
$c\in\calC$.

We need the following rather technical definition.
A list $c_1,\dotsc,c_{2r}$ of distinct elements of $\calC$ is a 
\demph{flat $i$-chain} if it
satisfies the following conditions:
\begin{enumerate}
  \item No $c_j$ is fixed by either $\varphi_{i-2}$ or $\varphi_i$ and each of  
        $c_3,c_5,\dotsc,c_{2r-1}$ is fixed by $\varphi_{i-1}$, 
  \item for each $j=1,\dotsc,r$, we have that $\varphi_{i}(c_{2j-1})=c_{2j}$, and 
  \item for each $j=1,\dotsc,r{-}1$ we have that $c_{2j+1}$ is equal to 
         $(\varphi_{i-2}\circ\varphi_{i-1})^t\circ\varphi_{i-2}(c_{2j})$, for some $t\geq
         0$. 
\end{enumerate}
 
\begin{definition}\label{D:dual_equivalence}
  Suppose that we have a family of combinatorial objects $\calC$, an integer $n$, and a
  descent statistic $\Des$ as before.
  A family of involutions $\varphi_i$ for $2\leq i\leq n{-}1$ on $\calC$ is a 
  \demph{dual equivalence} if it satisfies Conditions (i), (ii), and (iii) of
  Definition~\ref{D:strong_dual_equivalence} and if the following three conditions hold.
 \begin{enumerate}
  \item[(iv.a)]  For any $i<j\leq i+2$ and $c\in\calC$, the restricted generating function 
\[
     \sum_{a\in [c]_{(i,j)}} Q_{\Des_{(i,j)}(a)}
\]
     is symmetric and Schur-positive.

  \item[(iv.b)]  For every $3\leq i\leq n{-}1$ and $c\in\calC$ that is not fixed by
    $\varphi_i$ for which neither $c$ nor $\varphi_i(c)$ is  fixed by $\varphi_{i-1}$ or
    by $\varphi_{i+1}$, the quasisymmetric functions $Q_{[c]_{(i-1,i)}}$ and
    $Q_{[c]_{(i,i+1)}}$ are equal. 
 
  \item[(iv.c)]  For each $3<i<n{-}1$ and every flat $i$-chain $c_1,\dotsc,c_{2r}$, 
    if for some $j$ with $1<j<r$ we have that neither $c_{2j-1}$  nor $c_{2j}$ is fixed by 
    $\varphi_{i+1}$, then either none of $c_1,...,c_{2j}$ are fixed by $\varphi_{i+1}$ or
    none of $c_{2j-1},c_{2j},..., c_{2r}$ are fixed by $\varphi_{i+1}$.

    We also require the symmetric statement given by reversal, replacing $i$ by $n{-}i$.

 \end{enumerate}
\end{definition}

We state the main result of~\cite{A_TAMS}.

\begin{theorem}[\cite{A_TAMS}, Theorem~5.3]\label{Th:WDE_SP}
   If\/ $\{\varphi_i\mid 1<i<n\}$ is a dual equivalence for $(\calC,\Des)$, then the
   generating function $K_{(\calC,\Des)}$ is symmetric and Schur-positive. 
\end{theorem}

\begin{remark}\label{R:WDE_Combinatorial}
 The proof in~\cite{A_TAMS} invokes an algorithm to transform a dual
 equivalence into a strong dual equivalence, yielding an explicit, albeit complicated,
 combinatorial formula for the Schur coefficients in $K_{(\calC,\Des)}$ based on the
 description given for the Schur coefficients of a strong dual
 equivalence. 
 Conditions (iv.b) and (iv.c) are needed for the transformation algorithm.
 They will be explained in Subsection~\ref{SS:iv.c} in simple graphical terms.
\end{remark}

\begin{remark}\label{R:local}
 These conditions for a dual equivalence are local in that they depend only on the
 involutions $\varphi_i,\varphi_j$ for $|i{-}j|\leq 3$.
 This will enable us to reduce their verification to a computer check.
\end{remark}

This definition of a dual equivalence is seemingly more restrictive than given 
in~\cite{A_TAMS}; there, the second assertion about reversal symmetry in (iv.c) was not
invoked. 
While this additional condition can be shown to follow from the others, that is not
necessary for our purposes, for Theorem~\ref{Th:WDE_SP} (whose proof did not use this 
symmetry assertion) remains valid.

\begin{remark}\label{R:symmetry}
 Given a set $(\calC,\Des)$ with $\calC$ finite, $\Des(c)\subset[n{-}1]$ for $c\in\calC$,
 and involutions $\varphi_i$ on $\calC$ for $2\leq i\leq n{-}1$ that form a dual
 equivalence, we may construct three additional dual equivalences as follows.
 Let $\omega$ be the map that takes a subset of $[n{-}1]$ to its complement.
 Thus $i\in\omega(D)\Leftrightarrow i\not\in D$.
 The definition of a dual equivalence implies that the involutions
 $\varphi_i$ form a dual equivalence for $(\calC,\omega\circ\Des)$.

 Let $\rho$ be the involution on subsets $D$ of $[n{-}1]$ that reverses a subset
 \[
    \rho(D)\ :=\ \{n{-}j\mid j\in D\}\,.
 \]
 If we define $\psi_i:=\varphi_{n+1-i}$ for $2\leq i\leq n{-}1$, then 
 the definition of a dual equivalence implies 
 that these involutions $\psi_i$ form a dual 
 equivalence for $(\calC,\rho\circ\Des)$.
 This requires the symmetry assertion from (iv.c).

 We may compose these to get a third additional dual equivalence: the involutions
 $\psi_i$ form a dual equivalence for $(\calC,\rho\circ\omega\circ\Des)$.
\end{remark}

Given a finite set $(\calC,\Des)$ equipped with a notion of descent
$\Des(c)\subset[n{-}1]$ for $c\in\calC$ and a collection of involutions $\varphi_i$ on
$\calC$ for $2\leq i\leq n{-}1$, construct a colored graph $\calG_\calC$ with
vertex set $C$ which has an edge of color $i$ between $c$ and $\varphi_i(c)$ when
$c\neq\varphi_i(c)$. 
When the involutions $\varphi_i$ form a dual equivalence, $\calG_\calC$ is a 
\demph{dual equivalence graph}.

For example, here are the dual equivalence graphs for the partitions $(3,2)$
(from~\eqref{Eq:32_DE}) and $(3,1,1)$, which give the symmetric functions $s_{(3,2)}$ and
$s_{(3,1,1)}$, respectively.
 \begin{equation}\label{Eq:first_DEG}
  \begin{picture}(122,20)
   \put(0,7){\includegraphics{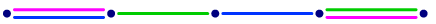}}

   \put( 14, 1){\MyCyan{\scriptsize$2$}}
   \put( 14,13){\MyMag{\scriptsize$3$}}
   \put( 44,12){\MyGreen{\scriptsize$4$}}
   \put( 74,12){\MyCyan{\scriptsize$2$}}
   \put(104,13){\MyGreen{\scriptsize$4$}}
   \put(104, 1){\MyMag{\scriptsize$3$}}
  \end{picture}
  \qquad \qquad
  \begin{picture}(122,25)(0,2)
   \put(0,2){\includegraphics{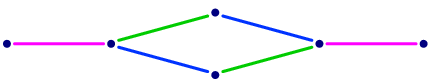}}

   \put( 14,14.5){\MyMag{\scriptsize$3$}}
   \put( 42,19){\MyGreen{\scriptsize$4$}}
   \put( 42, 1){\MyCyan{\scriptsize$2$}}
   \put( 75,19.5){\MyCyan{\scriptsize$2$}}
   \put( 75, 1){\MyGreen{\scriptsize$4$}}
   \put(104,14.5){\MyMag{\scriptsize$3$}}
  \end{picture}
 \end{equation}
Interchanging the labels $2\leftrightarrow 4$ is an automorphism
of both graphs. 
This may be understood from the interplay of the involution $\rho$ with quasisymmetric
functions.

Consider two involutions on the algebra of quasisymmetric functions.
The first, $\rho$, comes from reversing the alphabet. 
If $X=x_1<x_2<\dotsb$ and $\rho(X):=x_1>x_2>\dotsb$ is its reversal, then 
\[
   \rho(Q_D)(X)\ :=\ Q_D(\rho(X))\ =\ Q_{\rho(D)}(X)\,.
\]
(Here, we are using that the definition of quasisymmetric function relies on a total
ordering of the variables.)
Reversal restricts to the identity on symmetric functions.
Thus if $(\calC,\Des)$ admits a dual equivalence structure then
\[
   K_{(\calC,\Des)}\ =\ \rho(K_{(\calC,\Des)})\ =\ K_{(\calC,\rho\circ\Des)}\,,
\]
 where the first equality is because $K_{(\calC,\Des)}$ is symmetric and the second follows from the
 definitions of $\rho$.

 Similarly, the complementation map $\omega$ on subsets $D$ of $[n{-}1]$ induces an involution, 
 also written $\omega$, on quasisymmetric functions,
\[   
   \omega(Q_D)\ :=\ Q_{\omega(D)}\,.
\]
 This restricts to the fundamental involution on the algebra of symmetric functions,
\[
  \omega(s_\lambda)\ =\ s_{\lambda'}\,,
\]
where \defcolor{$\lambda'$} is the matrix transpose (interchanging rows with columns) of the
partition $\lambda$.
   
We suppressed the descent sets in these graphs~\eqref{Eq:first_DEG}.
This results in only a mild ambiguity.

\begin{lemma}\label{L:descentsDetermined}
 The edges in a connected component of a graph $\calG_\calC$ given by involutions that satisfy
 Conditions {\rm (i)} and {\rm (ii.b)} determine the descent sets
 of the vertices, up to a global application of the complementation map $\omega$.
\end{lemma}

\begin{proof}
 Let $v$ be a vertex in a connected component of $\calG_\calC$.
 Define a subset $D(v)\subset [n{-}1]$ as follows.
 First, let $1\in D(v)$.
 Then $2\in D(v)$ if and only if $v$ has no incident edge labeled 2.
 This ensures that Condition (i) of Definition~\ref{D:strong_dual_equivalence}
 holds. 
 For each of $i=3,\dotsc,n{-}1$ note that knowing if $i{-}1\in D(v)$ and whether or not
 $v$ has an incident edge with label $i$ determines whether or not $i\in D(v)$, by
 Condition (i).
 Then Conditions (i) and (ii.b) also determine $D(w)$ for any vertex $w$ adjacent to $v$.
 In this manner, we determine sets $D(w)$ for all vertices in the component of
 $\calG_\calC$ containing $v$.

 Had we assumed that $1\not\in D(v)$, we would have obtained the sets $\omega(D(w))$, the
 complement in $[n{-}1]$ of the sets $D(w)$.
\end{proof}

\begin{remark}
 Given a labeled graph $\calG$ on a set $(\calC,\Des)$, its reversal is the labeled graph
 $\rho\calG$ on the set $(\calC,\rho\circ\Des)$ where we replace an edge label of $i$ by $n-i$.
 Write $\omega\calG$ for the graph obtained from $\calG$ by replacing $\Des$ by $\omega\circ\Des$.

 We say that two labeled graphs $\calG$ and $\calG'$ are \demph{$\{\omega,\rho\}$-equivalent} if, as
 labeled graphs whose vertices are descent sets, we have
 \[
    \calG'\ \in\ \{ \calG\,,\ \omega\calG\,,\ \rho\calG\,,\ 
                   \rho\omega\calG=\omega\rho\calG\}\,.
 \]
 If $\calG,\calG'$ are equivalent graphs, then one is a dual equivalence graph if and only
 if the other one is.
\end{remark}

\subsection{Schubert vs.\ Schur constants}

Let $S_\infty$ be the set of permutations $w$ of $\{1,2,\dotsc\}$ with
$\{i\mid w(i)\neq i\}$ finite, which is the union of all finite symmetric groups.
The length $\ell(w)$ of a permutation $w\in S_\infty$ is its number of inversions, 
$\{i<j\mid w(i)>w(j)\}$.
We write \defcolor{$t_{ij}$} for the transposition interchanging the numbers $i<j$.

Lascoux and Sch\"utzenberger~\cite{LS82} defined the Schubert polynomials
$\frakS_w\in\Z[z_1,z_2,\dotsc]$, which are indexed by permutations $w\in S_\infty$.
The polynomial $\frakS_w$ is homogeneous of degree $\ell(w)$, and if
$w$ has no descents after position $n$, then
$\frakS_w\in\Z[z_1,\dotsc,z_n]$.
Schubert polynomials form a basis for the free $\Z$-module of all polynomials, which 
contains all Schur symmetric polynomials $s_\lambda(z_1,\dotsc,z_k)$ for all partitions
$\lambda$ with $\lambda_{k+1}=0$ and all $k$.
In particular, if \defcolor{$v(\lambda,k)$} is the permutation with a unique descent at
position $k$ and values $i+\lambda_{k+1-i}$ at $1\leq i\leq k$, then 
$\frakS_{v(\lambda,k)}=s_\lambda(z_1,\dotsc,z_k)$.
See~\cite{Mac91,Mac95} for more on Schubert and Schur polynomials.

Lascoux and Sch\"utzenberger showed that Schubert polynomials are polynomial
representatives of Schubert classes in the cohomology of the flag manifold.
Thus the integral structure constants
$c^w_{u,v}\in\Z$ defined by the identity
\[
   \frakS_u\cdot\frakS_v\ =\ \sum_w c^w_{u,v}\frakS_w
\]
are nonnegative.
We are concerned here with the \demph{Schubert vs.\ Schur constants $c^w_{u,v(\lambda,k)}$}
which are defined by the identity 
 \begin{equation}\label{Eq:SchubSchur}
   \frakS_u\cdot s_\lambda(z_1,\dotsc,z_k) =\ 
   \frakS_u\cdot \frakS_{v(\lambda,k)}\ =\ \sum_w c^w_{u,v(\lambda,k)}\frakS_w\,.
 \end{equation}
We use abstract dual equivalence to  give a combinatorial proof that these constants
$c^w_{u,v(\lambda,k)}$ are nonnegative.  

\begin{theorem}\label{Th:Main}
  For any permutations $u,w\in S_\infty$, positive integer $k$, and partition $\lambda$, we
  have $c^w_{u,v(\lambda,k)}\geq 0$.
\end{theorem}

This follows from Theorem~\ref{Th:DualEquivalance} in Section~\ref{S:two}.

Nonnegativity is known combinatorially in some cases.
When $u=v(\mu,k)$ is also Grassmannian with descent $k$, then 
$c^w_{v(\mu,k),v(\lambda,k)}=0$ unless $w=v(\nu,k)$, and in that case it equals the
Littlewood-Richardson coefficient $c^\nu_{\mu,\lambda}$.
Monk's formula (proven using geometry by Monk~\cite{Monk} and
combinatorially by Lascoux and Sch\"utzenberger~\cite{LS85}) is
 \begin{equation}\label{Eq:Monk}
   \frakS_u\cdot (z_1+\dotsb+z_k)\ =\ 
   \frakS_u\cdot \frakS_{t_{k\,k+1}}  \ =\   
   \sum \frakS_{u'}\,,
 \end{equation}
the sum over all $u'=ut_{ij}$ where $\ell(u')=\ell(u)+1$ and $i\leq k<j$.
The indices of summation in this formula define the cover relation in the 
\demph{$k$-Bruhat order, $\leq_k$}, that is,
\[
   u\ \lessdot_k\ u t_{ij}
   \qquad\mbox{if and only if }\ \ell(u t_{ij})=\ell(u)+1
   \quad\mbox{and}\quad i\leq k<j\,.
\]

Given a cover $u\lessdot_kw=ut_{ij}$ in the $k$-Bruhat order, let $t_{ab}$ with $a<b$ be 
the transposition such that $w=t_{ab}u$.
Label the cover $u\lessdot_k t_{ab}u$ with the integer $b$.
This gives the $k$-Bruhat order the structure of a labeled poset.
Write $u\xrightarrow{\,b\,}w$ to indicate that $u\lessdot_k w$ where $wu^{-1}$ is a
transposition $t_{ab}$ with $a<b$ so that the cover is labeled with $b$.
Bergeron and Sottile~\cite{BS_Hopf} considered the quasisymmetric generating function for
intervals in this labeled poset,
\[
   K_{[u,w]_k}\ :=\ 
    \sum_{\mbox{\scriptsize chains }c\mbox{\scriptsize\ in\ }[u,w]_k} Q_{\Des(c)}\,.
\]
Using the Pieri formula (see below) and the Jacobi-Trudy formula, they showed that
this is the symmetric generating function of the Schubert vs.\ Schur coefficients, 
 \begin{equation}\label{Eq:Skew_Schubert}
  K_{[u,w]_k}\ =\ 
   \sum_{|\lambda|=\ell(w)-\ell(u)}  c^w_{u,v(\lambda,k)} s_\lambda\,.
 \end{equation}
Our proof of Theorem~\ref{Th:Main} will involve putting the structure of a dual
equivalence on the set of labeled chains in an interval in the $k$-Bruhat order.
When $u,w$ are Grassmannian with descent $k$ and correspond to the partitions
$\mu,\lambda$, respectively, then the interval $[u,w]_k$ is isomorphic to the
interval $[\mu,\lambda]$ in Young's lattice, if we shift the labels by $-k$.

We explain the formula~\eqref{Eq:Skew_Schubert} and develop some combinatorics of the
$k$-Bruhat order.
Recall~\cite[I.5, Example 2]{Mac95} that we have 
\[
   (z_1+\dotsb+z_k)^m\ =\ \sum_{|\lambda|=m} f^\lambda s_\lambda(z_1,\dotsc,z_k)\,,
\]
the sum over all partitions $\lambda$ of $m$ where $f^\lambda$ is the number of standard
Young tableaux of shape $\lambda$.
Multiplying this by $\frakS_u$, expanding using~\eqref{Eq:SchubSchur} and~\eqref{Eq:Monk}, 
and then equating coefficients of $\frakS_w$ gives the following proposition.

\begin{proposition}[Prop.\ 1.1.1~\cite{BS_Duke}]
 The number of saturated chains in the $k$-Bruhat order from $u$ to $w$ is equal to
\[
   \sum_\lambda f^\lambda c^w_{u,v(\lambda,k)}\,.
\]
\end{proposition}   

In particular, $c^w_{u,v(\lambda,k)}=0$ unless $u\leq_k w$.
This suggests that a description of the constants in terms of chains in the $k$-Bruhat order
should exist.
Our proof of Theorem~\ref{Th:Main} using abstract dual equivalence provides such a
description. 

The complete homogeneous symmetric polynomial $h_m(z_1,\dotsc,z_k)$ is the sum of all
monomials of degree $m$ in $z_1,\dotsc,z_k$.
It is the Schur polynomial $s_{(m)}(z_1,\dotsc,z_k)$ where $(m)$ is the partition
with only one part and it has size $m$.
The Pieri formula gives the rule for multiplication of a Schubert polynomial by
$h_m(z_1,\dotsc,z_k)$.
It is
\[
   \frakS_u\cdot h_m(z_1,\dotsc,z_k)\ =\ \sum_\gamma \frakS_{\ep(\gamma)}\,,
\]
the sum over all saturated chains $\gamma$ in the $k$-Bruhat order of length $m$
starting at $u$, 
\[
   \gamma\ \colon\ 
    u  \ \xrightarrow{\,b_1\,}\ 
    w_1\ \xrightarrow{\,b_2\,}\ 
    \dotsb\ 
    \xrightarrow{\,b_m\,}\ w_m\ =:\ \ep(\gamma)
\]
with increasing labels, $b_1<b_2<\dotsb<b_m$.
Call this an \demph{increasing chain} of length $m$.

An equivalent formula was conjectured by Lascoux and Sch\"utzenberger in~\cite{LS82}, where
they suggested an inductive proof.
The above formulation of the Pieri formula was conjectured in~\cite{BeBi}, and proven using
geometry in~\cite{PieriSchubert}. 
Postnikov gave a combinatorial proof~\cite{Postnikov} and later another
combinatorial proof~\cite{K_Pieri} was given following the suggestion of Lascoux and
Sch\"utzenberger.

The Pieri formula implies a formula for multiplying a Schubert polynomial by a product of
complete homogeneous symmetric polynomials in terms of chains whose labels have descents
in a subset.
Through the Jacobi-Trudy formula, it implies a (non-positive) combinatorial formula for the 
constants $c^{w}_{u,v(\lambda,k)}$ in terms of chains in $[u,w]_k$~\cite{BS_Skew}.
There, the symmetric function on the right hand side of~\eqref{Eq:Skew_Schubert} was
defined, and in~\cite{BS_Hopf} the quasisymmetric function $K_{[u,w]_k}$ was defined and
the equality~\eqref{Eq:Skew_Schubert} was proven.

\subsection{Grassmannian-Bruhat order on $S_\infty$}

The Grassmannian-Bruhat order $\preceq$ was introduced and studied in~\cite[\S 3]{BS_Duke}.
It is a ranked labeled poset on the infinite symmetric group $S_\infty$ which has the following
defining property:
If $u\leq_k w$ and we set $\zeta:=wu^{-1}$, then the map $\eta\mapsto \eta u$ is an isomorphism
of labeled posets,
\[
   [u,w]_k\ \xrightarrow{\ \sim\ }\ [e,\zeta]_{\preceq}\,.
\]
Consequently, the quasisymmetric function $K_{[u,w]_k}$ and therefore the constants 
$c^w_{u,v(\lambda,k)}$ depend only upon $\zeta$ and $\lambda$.
We say that $\zeta$ has \demph{rank $n$} if the poset $[e,\zeta]_\preceq$ has rank $n$.

For example, the  labeled poset of  Example~\ref{Ex:143256} is isomorphic
to $[142635,456123]_3$.
The computation in that example shows that $K_{(1,4,5,3,2,6)}=s_{32}+s_{311}$ and therefore
\[
   c^{456123}_{142635,v((3,2),3)}\ =\ 
   c^{456123}_{142635,v((3,1,1),3)}\ =\ 1
   \quad\mbox{and}\quad
   c^{456123}_{142635,v((2,2,1),3)}\ =\ 0\,.
\]

The Grassmannian-Bruhat order has a more direct definition, given in Theorem 3.1.5({\it ii})
of~\cite{BS_Duke}.
For $\zeta\in S_\infty$, define 
$\defcolor{\up(\zeta)}:=\{a\mid a<\zeta(a)\}$ and 
$\defcolor{\dw(\zeta)}:=\{b\mid b>\zeta(b)\}$.
Let $\eta,\zeta\in S_\infty$.
Then $\eta\preceq\zeta$ if and only if the following hold.
\begin{enumerate}
 \item For all $a\in\up(\zeta)$, we have $a\leq\eta(a)\leq\zeta(a)$,
 \item for all $b\in\dw(\zeta)$, we have $b\geq\eta(b)\geq\zeta(b)$, and 
 \item For all $a<b$ with either $a,b\in\up(\zeta)$ or  $a,b\in\dw(\zeta)$,
       if $\zeta(a)<\zeta(b)$, then $\eta(a)<\eta(b)$.
\end{enumerate}
The labeled structure of $\preceq$ is inherited from the $k$-Bruhat order as follows.
When $\eta\precdot\zeta$ is a cover, we have a transposition $t_{ab}=\zeta\eta^{-1}$ with
$a<b$, and the cover is labeled with $\max\{a,b\}$.
Saturated chains in an interval $[e,\zeta]_\preceq$ in the Grassmannian-Bruhat order are
therefore represented by sequences $(t_{a_1 b_1},\dotsc,t_{a_n b_n})$ of transpositions, where,
if we set $\eta_i := t_{a_i b_i}\dotsb t_{a_1 b_1}$ for each $i=1,\dotsc,n$ with $e=\eta_0$, 
then we have $\eta_{i-1}\precdot \eta_i$ for each $i=1,\dotsc,n$ and $\eta_n=\eta$.

There are several important properties that can be deduced
from this definition.
One is \demph{invariance under relabeling}.
Suppose that $I=\{i_1<i_2<\dotsb\}\subset\N$ is a set of integers.
This induces an inclusion of $S_{|I|}$ into $S_\infty$,
\[
   \iota_I(\zeta) (j)\ =\ \left\{
    \begin{array}{ccl} j&\mbox{ if }&j\not\in I\\
                       i_{\zeta(k)}&\mbox{ if }&j=i_k\in I
    \end{array}\right.\ .
\]
For example, if $\zeta=(1,4,2)(3,5)$ and $I=\{1,3,5,7,9,\dotsc\}$, then 
$\iota_I(\zeta)=(1,7,3)(5,9)$.
Then $\eta\preceq\zeta$ if and only if $\iota_I(\eta)\preceq\iota_I(\zeta)$.
The following properties of the Grassmannian-Bruhat order
from Theorem 3.2.3 of~\cite{BS_Duke} follow directly from its definition.

\begin{proposition}\label{P:properties}
 The Grassmannian-Bruhat order $\preceq$ on $S_\infty$ has the following properties.
 \begin{enumerate}
  \item If $\eta\preceq\zeta$, then $\xi\mapsto \xi\eta^{-1}$ is an isomorphism of labeled
    posets $[\eta,\zeta]_\preceq\xrightarrow{\;\sim\;}[e,\zeta\eta^{-1}]_{\preceq}$.

  \item For every $I\subset \N$ the map $\iota_I\colon S_{|I|}\to S_\infty$ is an injection
    of labeled posets such that if $\eta\in S_{|I|}$ then 
    $[e,\eta]_\preceq$ is label-equivalent to $[e,\iota(\eta)]_{\preceq}$.

  \item The map $\eta\mapsto \eta\zeta^{-1}$ induces an order-reversing but label-preserving
    isomorphism between $[e,\eta]_\preceq$ and $[e,\eta^{-1}]_\preceq$.
 \end{enumerate}  
\end{proposition}

\begin{remark}\label{R:GBO}
 This result has consequences for the study of chains in the Grassmannian-Bruhat order, and
 ultimately the quasisymmetric function $K_{[u,w]_k}$.
 By (1), if
\[
   \eta\ \precdot\ t_{a_1 b_1}\eta\ \precdot\ t_{a_2 b_2} t_{a_1 b_1}\eta\ 
   \precdot\ \dotsb\ \precdot\ 
   t_{a_n b_n}\dotsb t_{a_1 b_1}\eta\ =\ \zeta
\]
 is a (saturated) chain in the interval $[\eta,\zeta]_\preceq$, then
 \begin{equation}\label{Eq:rooted_chain}
   e\ \precdot\ t_{a_1 b_1}\ \precdot\ 
     t_{a_2 b_2} t_{a_1 b_1}\precdot\ \dotsb\ \precdot\ 
   t_{a_n b_n}\dotsb t_{a_2 b_2} t_{a_1 b_1}\ =\ \zeta\eta^{-1}\ =:\ \xi
 \end{equation}
 is a chain in the interval $[e,\xi]_\preceq$.

 Statement (3) implies that~\eqref{Eq:rooted_chain} is a chain in $[e,\xi]_\preceq$ if and 
 only if
\[
   e\ \precdot\ t_{a_n b_n}\ \precdot\ \dotsb\  \precdot\ 
   t_{a_2 b_2}\dotsb t_{a_n b_n}\ \precdot\ 
   t_{a_1 b_1} t_{a_2 b_2} \dotsb t_{a_n b_n}\ =\ \xi^{-1}
\]
 is a chain in $[e,\xi^{-1}]_\preceq$.
 Let $c$ be the chain~\eqref{Eq:rooted_chain} in $[e,\xi]_\preceq$ and $c^\vee$ the
 corresponding reversed chain in $[e,\xi^{-1}]_\preceq$.
 Observe that the sequence of labels in reversed in passing from $c$ to $c^\vee$, and therefore
 $\Des(c^\vee)=\rho\omega\Des(c)$.
 That is, $i\in\Des(c)\Leftrightarrow n{-}i\not\in\Des(c^\vee)$.
\end{remark}

Chains in the Grassmannian-Bruhat order were studied in~\cite{Monoide}.
It was written in terms of a monoid for that order analogous to the nil-Coxeter monoid for the
weak order on the symmetric group.
We summarize its results for chains in the Grassmannian-Bruhat order.

\begin{proposition}\label{P:substitutions}
 Let $c:=(t_{a_1 b_1},\dotsc,t_{a_n b_n})$ be the sequence of transpositions in a chain in
 the interval $[e,\zeta]_\preceq$ of the Grassmannian-Bruhat order of rank $n$.
 Any other chain in this interval is obtained from $c$ by a sequence of
 substitutions that replace a subchain of length two or three by an equivalent subchain 
 according to one of the following rules.
 \begin{equation} \label{Eq:ChainRelations}
   \begin{array}{lrcll}
     (i)\quad& (t_{\beta\gamma},t_{\gamma\delta},t_{\alpha\gamma})&\leftrightarrow&
               (t_{\beta\delta},t_{\alpha\beta},t_{\beta\gamma})&
               \quad\mbox{if }\alpha<\beta<\gamma<\delta\\
     (ii)\quad& (t_{\alpha\gamma},t_{\gamma\delta},t_{\beta\gamma})&\leftrightarrow&
                (t_{\beta\gamma},t_{\alpha\beta},t_{\beta\delta})&
                \quad\mbox{if }\alpha<\beta<\gamma<\delta\\
     (iii)\quad& (t_{\alpha\beta},t_{\gamma\delta})&\leftrightarrow&
                 (t_{\gamma\delta},t_{\alpha\beta})&
          \quad\mbox{if }\beta<\gamma\mbox{ or }\alpha<\gamma<\delta<\beta
   \end{array}
 \end{equation}
 Moreover, any sequence of such replacements results in a chain in the interval
 $[e,\zeta]_\preceq$. 

 No chain in $[e,\zeta]_\preceq$ contains a subsequence having
 the following forms, 
 \begin{equation} \label{Eq:ZeroRelations}
   \begin{array}{lrcll}
     (iv)\quad& t_{\alpha\gamma},t_{\beta\delta}&\mbox{or}&
               t_{\beta\delta},t_{\alpha\gamma}&
          \quad\mbox{for }\alpha\leq\beta<\gamma\leq\delta\\
     (v)\quad& t_{\beta\gamma},t_{\alpha\beta},t_{\beta\gamma}&\mbox{or}&
               t_{\alpha\beta},t_{\beta\gamma},t_{\alpha\beta}&
               \quad\mbox{for }\alpha<\beta<\gamma\,.
   \end{array}
 \end{equation}
 Finally, a sequence $(t_{a_1 b_1},\dotsc,t_{a_n b_n})$ of transpositions is a chain in the
 Grassmannian-Bruhat order if and only if for any sequence of substitutions using $(i)$,
 $(ii)$, or $(iii)$, no subsequence of form $(iv)$ or $(v)$ is ever encountered.
\end{proposition}

Observe that the reversal of chains (sending a chain in $[e,\zeta]_\preceq$ to the
corresponding chain in $[e,\zeta^{-1}]_\preceq$) fixes the substitution $(iii)$ and
interchanges $(i)$ with $(ii)$.

\begin{remark}\label{disjoint}
 We express the substitution $(iii)$ in another form.
 Call transpositions satisfying the conditions for $(iii)$ \demph{disjoint}.
 There are exactly two pairs of disjoint transpositions on a four element set 
 $\{\alpha<\beta<\gamma<\delta\}$, namely
 \[
   (t_{\alpha\beta},t_{\gamma\delta})
   \quad\mbox{and}\qquad
   (t_{\alpha\delta},t_{\beta\gamma})\,.
 \]
 These correspond to the two noncrossing matchings on this set.
 The one crossing matching gives rise to the excluded
 subsequence~\eqref{Eq:ZeroRelations}$(iv)$. 
 We display these three below.
 \[
  \begin{picture}(60,27)(0,-10)
   \put(0,0){\includegraphics{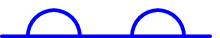}}
   \put( 4,-8){$\alpha$}  \put(19,-10){$\beta$}
   \put(34,-8){$\gamma$}  \put(49,-10){$\delta$}
  \end{picture}
   \qquad
  \begin{picture}(60,27)(0,-10)
   \put(0,0){\includegraphics{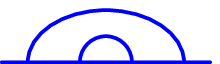}}
   \put( 4,-8){$\alpha$}  \put(19,-10){$\beta$}
   \put(34,-8){$\gamma$}  \put(49,-10){$\delta$}
  \end{picture}
   \qquad
  \begin{picture}(60,27)(0,-10)
   \put(0,0){\includegraphics{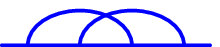}}
   \put( 4,-8){$\alpha$}  \put(19,-10){$\beta$}
   \put(34,-8){$\gamma$}  \put(49,-10){$\delta$}
  \end{picture}
 \]
 The pairs of transpositions $(t_{\alpha\beta},t_{\beta\gamma})$ and 
 $(t_{\beta\gamma},t_{\alpha\beta})$ with $\alpha<\beta<\gamma$ are \demph{connected} and
 they each give a valid chain of length 2.
 By~\eqref{Eq:ZeroRelations}$(iv)$, only connected or disjoint pairs of
 transpositions form valid chains of length 2.

 A set of $n$ pairwise disjoint transpositions with indices drawn 
 from a set of size $2n$ corresponds to a noncrossing complete matching on 
 this set, and there are therefore $C_n$ of these, where $C_n$ is the $n$th Catalan number. 
 Given such a matching, all $n!$ orderings of its $n$ transpositions obtained using
 substitution~\eqref{Eq:ChainRelations}$(iii)$ give valid chains.
\end{remark}

\begin{example}\label{Ex:124536}
The interval $[e,(124536)]_\preceq$ is displayed in Figure~\ref{F:124536} with its
nine chains.
\begin{figure}[htb]
 \begin{minipage}[b]{155pt}
   \begin{enumerate}
    \setcounter{enumiii}{2}
    \item $(t_{35},t_{56},t_{45},t_{24},t_{12})$\rule{0pt}{12pt}\vspace{9.1pt}
    \item $(t_{45},t_{34},t_{46},t_{24},t_{12})$\vspace{9.1pt}
    \item $(t_{45},t_{36},t_{23},t_{34},t_{12})$\vspace{9.1pt}
    \item $(t_{36},t_{45},t_{23},t_{34},t_{12})$\vspace{9.1pt}
    \item $(t_{45},t_{36},t_{23},t_{12},t_{34})$\vspace{9.1pt}
    \item $(t_{36},t_{45},t_{23},t_{12},t_{34})$\vspace{9.1pt}
    \item $(t_{36},t_{23},t_{45},t_{34},t_{12})$\vspace{9.1pt}
    \item $(t_{36},t_{23},t_{45},t_{12},t_{34})$\vspace{9.1pt}
    \item $(t_{36},t_{23},t_{12},t_{45},t_{34})$\vspace{5.1pt}
   \end{enumerate}
  \end{minipage}
  \begin{picture}(195,205)(-88,0)
   \put(-100,0){\includegraphics{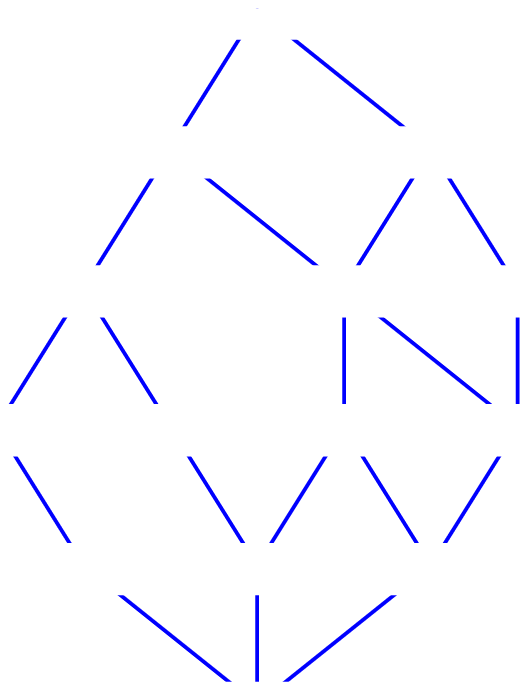}}

   \put(-22,197){$(124536)$}

   \put(-27,182){\smaller$t_{12}$}\put(27,182){\smaller$t_{34}$}

   \put(-44,157){$(24536)$}   \put(26,157){$(1236)(45)$}

   \put(-52,142){\smaller$t_{24}$}\put(1,142){\smaller$t_{34}$}
        \put(23,142){\smaller$t_{12}$}\put(64,142){\smaller$t_{45}$}

   \put(-66,117){$(3645)$} \put(1,117){$(326)(45)$}
         \put(59,117){$(1236)$}

   \put(-77,102){\smaller$t_{45}$}\put(-36,102){\smaller$t_{46}$}
    \put(10,102){\smaller$t_{23}$}\put(51,102){\smaller$t_{45}$}
                   \put(78,102){\smaller$t_{12}$}

   \put(-88,77){$(365)$} \put(-38,77){$(345)$} 
         \put(5,77){$(36)(45)$}   \put(63,77){$(236)$}

   \put(-76,57){\smaller$t_{56}$}\put(-25,57){\smaller$t_{34}$}
       \put(-2,62){\smaller$t_{36}$}\put(39,62){\smaller$t_{45}$}
                   \put(65,57){\smaller$t_{23}$}

   \put(-60,37){$(35)$}\put(-10,37){$(45)$}\put(40,37){$(36)$}

   \put(-38,17){\smaller$t_{35}$}\put(3,23){\smaller$t_{45}$}\put(30,17){\smaller$t_{36}$}

   \put(-3,-1){$e$}

  \end{picture}
 \caption{The interval $[e,(124536)]_\preceq$ and its nine chains.}\label{F:124536}
\end{figure}
By Proposition~\ref{P:substitutions}, these chains are connected by
substitutions of type $(i)$, $(ii)$, and $(iii)$,
We display the graph of these substitutions, with the numbering from
Figure~\ref{F:124536}. 
\[
  \begin{picture}(213,56)(0,-3)
  \put(6,2){\includegraphics{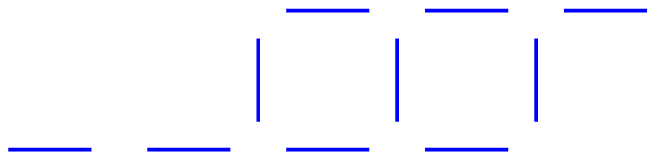}}
  \put(  0,0){$(1)$}
  \put( 40,0){$(2)$}
  \put( 80,0){$(3)$} \put( 80,40){$(5)$}
  \put(120,0){$(4)$} \put(120,40){$(6)$}
  \put(160,0){$(7)$} \put(160,40){$(8)$}
                     \put(200,40){$(9)$}

  \put( 21,7){\scriptsize$(ii)$}
  \put( 63,7){\scriptsize$(i)$}
   \put( 69,22){\scriptsize$(iii)$}
  \put(100,7){\scriptsize$(iii)$}  \put(100,47){\scriptsize$(iii)$}
   \put(109,22){\scriptsize$(iii)$}
  \put(140,7){\scriptsize$(iii)$}  \put(140,47){\scriptsize$(iii)$}
  \put(169.5,22){\scriptsize$(iii)$}
                                   \put(180,47){\scriptsize$(iii)$}
  \end{picture}
\]

The interval $[e,(145326)]_\preceq$ of Example~\ref{Ex:143256} has eleven chains.
If we number them 1---11 according to the order in which their labels appear
in~\eqref{Eq:ElevenChainLabels}, then we have the following graph of substitutions, which
may be read from the poset of Figure~\ref{F:145326}.
 \begin{equation}\label{Eq:SubstitutionGraph}
  \raisebox{-23pt}{\begin{picture}(233,56)(0,-3)
   \put(6,2){\includegraphics{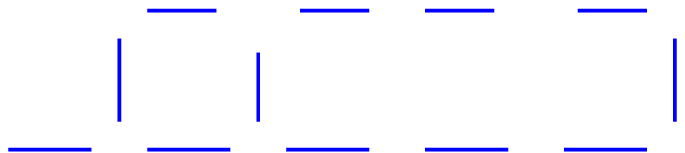}}
   \put(  0,0){$(6)$}
   \put( 40,0){$(7)$} \put( 40,40){$(9)$}
   \put( 80,0){$(1)$} \put( 78,40){$(10)$}
   \put(120,0){$(8)$} \put(120,40){$(5)$}
   \put(160,0){$(2)$} \put(158,40){$(11)$}
   \put(200,0){$(3)$} \put(200,40){$(4)$}

   \put( 23,7){\scriptsize$(i)$}
   \put( 61,7){\scriptsize$(ii)$}  \put( 61,47){\scriptsize$(ii)$}
   \put(100,7){\scriptsize$(iii)$} \put(101,47){\scriptsize$(ii)$}
   \put(140,7){\scriptsize$(iii)$} \put(140,47){\scriptsize$(iii)$}
   \put(182,7){\scriptsize$(ii)$}  \put(181,47){\scriptsize$(iii)$}
 
   \put( 29,22){\scriptsize$(iii)$}
   \put( 69,22){\scriptsize$(iii)$}
   \put(210,22){\scriptsize$(iii)$}
  \end{picture}}
 \end{equation}

\end{example}

\section{A dual equivalence for $K_{[e,\zeta]_\preceq}$}\label{S:two}

We prove Theorem~\ref{Th:Main} by putting the structure of a dual equivalence on the set
of labeled chains in an interval of the Grassmannian-Bruhat order. 
For this, we will use the substitutions of
Proposition~\ref{P:substitutions} to define involutions $\varphi_i$ 
on chains in the Grassmannian-Bruhat order, which we will show satisfy 
the conditions for a dual equivalence.
For the last three conditions, (iv.a), (iv.b), and (iv.c), we use
Proposition~\ref{P:properties}~(2) to reduce this to a finite verification that we
complete with the help of a computer.

Let $c$ be a chain in the Grassmannian-Bruhat order of length $n$.
Using the  substitutions $(i)$, $(ii)$, and $(iii)$ of
Proposition~\ref{P:substitutions} to define $\varphi_i(c)$ ensures that
it is a chain in the same interval as $c$.
To ensure that the involutions  $\varphi_i$ satisfy Condition
(iii) of a dual equivalence, we require that  $\varphi_i$ acts \demph{locally}
on $c$ in that the chains $c$ and $\varphi_i(c)$ agree, except possibly at positions
$i{-}1,i,i{+}1$. 
We ask for a \demph{uniform} definition of these involutions in that 
$\varphi_i$  and  $\varphi_j$ have the same definition, once we translate indices by
$j{-}i$, and that this definition is invariant under relabeling.
We explain this.
If $I\subset\N$ is a subset and $c$ a chain in $[e,\zeta]_\preceq$, write $\iota_I(c)$ for
the corresponding chain in $[e,\iota_I(\zeta)]_\preceq$.
If for every $I$, we have,
\[
   \iota_I(\varphi_i(c))\ =\ \varphi_i(\iota_I(c))\,,
\]
then $\varphi_i$ is \defcolor{invariant under relabeling}.
It follows that local and uniform involutions are determined by their action on chains of
length three.
An involution $\varphi_i(c)$ is \demph{reversible} if
$\varphi_i(c^\vee)=\varphi_i(c)^\vee$, where $c^\vee$ is the reversal of the chain $c$.

\begin{lemma}\label{L:locality}
 There is a unique dual equivalence on chains in the Grassmannian-Bruhat order 
 of length three for which $\varphi_i(c)$ is obtained from $c$ by a minimal number
 of substitutions $(i)$, $(ii)$, and $(iii)$ of Proposition~$\ref{P:substitutions}$.
 It is given by the following three rules.
\[
  \begin{array}{rclcl}
   \mbox{\bf (A)}&&
    \begin{array}{c}
      \varphi_2\colon
      (t_{\beta\gamma},t_{\alpha\beta},t_{\beta\delta})\ 
      \longleftrightarrow\  (t_{\alpha\gamma},t_{\gamma\delta},t_{\beta\gamma})\\
      \varphi_2\colon
      (t_{\beta\delta},t_{\alpha\beta},t_{\beta\gamma})\ 
      \longleftrightarrow\  (t_{\beta\gamma},t_{\gamma\delta},t_{\alpha\gamma})
     \end{array}
     &&\mbox{if\ }\alpha<\beta<\gamma<\delta\vspace{4pt}\\
   \mbox{\bf (B)}&&  
    \begin{array}{c}
      \varphi_2\colon
      (t_{b\beta},t_{a\alpha},t_{c\gamma})\ 
      \longleftrightarrow\  (t_{b\beta},t_{c\gamma},t_{a\alpha})\\
      \varphi_2\colon  (t_{c\gamma},t_{a\alpha},t_{b\beta})\ 
      \longleftrightarrow\  (t_{a\alpha},t_{c\gamma},t_{b\beta})
     \end{array}
     &&\hspace{-5pt}\begin{array}{l}\mbox{if\ }\alpha<\beta<\gamma\\
      \mbox{\ and\ }\{a,\alpha\},\ \{c,\gamma\}\mbox{ are disjoint}
     \end{array}\vspace{4pt}
  \\
   \mbox{\bf (C)}&&
    \begin{array}{c}
      \varphi_2\colon  (t_{pq},t_{\alpha\beta},t_{\beta\gamma})\ 
      \longleftrightarrow\  (t_{\alpha\beta},t_{\beta\gamma},t_{pq})
      \\
      \varphi_2\colon    (t_{\beta\gamma},t_{\alpha\beta},t_{pq})\ 
      \longleftrightarrow\  (t_{pq},t_{\beta\gamma},t_{\alpha\beta})
    \end{array}
     &&\mbox{if\ }\alpha<\beta<p<q<\gamma\\
   \end{array}
\]
%
%
%
 This is uniform and reversible.
\end{lemma}

\begin{proof}
 A dual equivalence on the set $C_3$ of chains of length three in the
 Grassmannian-Bruhat order is given by an involution $\varphi_2$ on $C_3$ satisfying
 Conditions (i) and (ii.b) of Definition~\ref{D:strong_dual_equivalence}, for the 
 other conditions are automatically satisfied on chains of length three.
 That is, if $c$ is a chain of length three in an interval $[e,\zeta]_\preceq$, then
 $\varphi_2(c)$ is another chain of length three in $[e,\zeta]_\preceq$, and we have 
\[
  \begin{array}{rcl}
    \mbox{\rm (i)}&\varphi_2(c)=c\ \Longleftrightarrow\ 
     \Des(c)=\emptyset \mbox{ or } \Des(c)=\{1,2\}.\\
    \mbox{\rm (ii.b)}&\varphi_2(c)\neq c\ \Longleftrightarrow\ 
     \Des(c)=\{1\} \mbox{ and } \Des(\varphi_2(c))=\{2\}
      \mbox{ or vice-versa}.
  \end{array}
\]
 In both cases, the quasisymmetric function of the chains matched by $\varphi_2$ is
 symmetric, so Condition (iv.a) is satisfied and the other conditions are vacuous for
 chains of length three.

 Ignoring labels, there are exactly five isomorphism classes of intervals in the
 Grassmannian-Bruhat order of rank three.
\[
  (a)\ \raisebox{-20pt}{\includegraphics{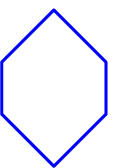}}\qquad
  (b)\ \raisebox{-20pt}{\includegraphics{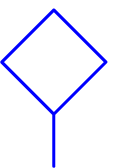}}\qquad
  (c)\ \raisebox{-20pt}{\includegraphics{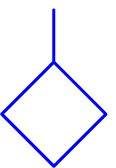}}\qquad
  (d)\ \raisebox{-20pt}{\includegraphics{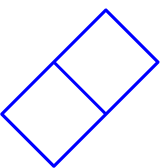}}\qquad
  (e)\ \raisebox{-20pt}{\includegraphics{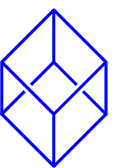}}
\]
 Intervals of type $(a)$ come from a substitution~$(i)$
   or~$(ii)$ from Proposition~\ref{P:substitutions}.
 Observe that if $\alpha<\beta<\gamma<\delta$, then 
\[
   \Des(t_{\alpha\gamma},t_{\gamma\delta},t_{\beta\gamma})\ =\ \{2\}
   \qquad\mbox{while}\qquad
   \Des(t_{\beta\gamma},t_{\alpha\beta},t_{\beta,\delta})\ =\ \{1\}
\]
 in the case of~$(i)$, and similarly
 for~$(ii)$. 
 It follows that $\varphi_2$ must interchange the two chains in intervals of type $(a)$, 
 which gives the rule {\bf (A)}.

 An interval of type $(b)$ will have two chains, necessarily of the form
 $(t_{pq},t_{\alpha\beta},t_{\gamma\delta})$ and $(t_{pq},t_{\gamma\delta},t_{\alpha\beta})$ where
 $t_{\alpha\beta}$ and $t_{\gamma\delta}$ are disjoint (so the diamond is formed by a
 substitution~$(iii)$), but neither is disjoint from $t_{pq}$. 
 To distinguish these chains, suppose that $\beta<\delta$.
 By the Prohibition~\eqref{Eq:ZeroRelations}$(iv)$, $p\in\{\beta,\delta\}$ and
 $q\in\{\alpha,\gamma\}$.
 As $p<q$, the only possibility is that $p=\beta$ and $q=\gamma$, and
 $\alpha<\beta<\gamma<\delta$.  
 Then the descent sets of the chains are
 $\Des(t_{\beta\gamma},t_{\alpha\beta},t_{\gamma\delta})=\{1\}$ 
 and $\Des(t_{\beta\gamma},t_{\gamma\delta},t_{\alpha\beta})=\{2\}$, and we must have 
\[
   \varphi_2\ \colon\  (t_{\beta\gamma},t_{\alpha\beta},t_{\gamma\delta})
    \ \longleftrightarrow\ (t_{\beta\gamma},t_{\gamma\delta},t_{\alpha\beta})\,.
\]
 Similarly, there is one choice for the chains in an interval of type $(c)$, and we must
 have  
\[
   \varphi_2\ \colon\  (t_{\gamma\delta},t_{\alpha\beta},t_{\beta\gamma})
    \ \longleftrightarrow\ (t_{\alpha\beta},t_{\gamma\delta},t_{\beta\gamma})\,.
\]
 These are both included in rule {\bf (B)}.

 An interval of type $(d)$ has one transposition disjoint from the other two
 (both diamonds are formed by substitutions~$(iii)$), but the
 other two are connected (do not commute), and have either the form
 $t_{\alpha\beta},t_{\beta\gamma}$ or the form $t_{\beta\gamma},t_{\alpha\beta}$ with 
 $\alpha<\beta<\gamma$.
 We consider chains whose connected permutations have the first type.

 The third transposition is $t_{pq}$ with one of the five inequalities holding:
 \begin{eqnarray*}
  &{\rm (1)}\  q<\alpha\,, \qquad
   {\rm (2)}\ \alpha<p<q<\beta\,, \qquad
   {\rm (3)}\ \beta<p<q<\gamma\,,& \\
  &{\rm (4)}\ \gamma<p\,,\qquad\mbox{or}\qquad
   {\rm (5)}\ p<\alpha\,,\ \gamma<q\,.&
 \end{eqnarray*}
 The three chains in this interval will be 
 $(t_{pq},t_{\alpha\beta},t_{\beta\gamma})$, 
 $(t_{\alpha\beta},t_{pq},t_{\beta\gamma})$, and 
 $(t_{\alpha\beta},t_{\beta\gamma},t_{pq})$,
 with each obtained from the previous by one application of the
 substitution~$(iii)$. 
 For each of the five types of chains, the three descent sets are as follows, respectively
 \begin{eqnarray*}
  &{\rm (1)}\   \emptyset,\{1\},\{2\}\,, \qquad
   {\rm (2)}\  \emptyset,\{1\},\{2\}\,, \qquad
   {\rm (3)}\ \{1\},\emptyset,\{2\}\,,& \\
  &{\rm (4)}\  \{1\},\{2\},\emptyset\,,\qquad\mbox{or}\qquad
   {\rm (5)}\   \{1\},\{2\},\emptyset\,.&
 \end{eqnarray*}
 For all types except (3), the chains with descent sets $\{1\}$ and $\{2\}$ are obtained from
 each other by a single substitution~$(iii)$, and these are again
 covered by rule {\bf (B)}. 
 For the remaining type (3), $\varphi_2$ must involve two applications of the
 substitution~$(iii)$, so that 
 \[
  \varphi_2\ \colon\ 
   (t_{pq},t_{\alpha\beta},t_{\beta\gamma})
    \ \longleftrightarrow\ 
   (t_{\alpha\beta},t_{\beta\gamma},t_{pq})\,,
 \]
 and this is covered by rule {\bf (C)}.
 The case when the connected permutations are $t_{\beta\gamma},t_{\alpha\beta}$ is identical
 upon reversing the chain.
 
 In the remaining case of an interval of type $(e)$, the three transpositions are
 disjoint. 
 Writing them as $t_{a\alpha}$, $t_{b\beta}$, and $t_{c\gamma}$ with
 $\alpha<\beta<\gamma$, the six chains are related by
 substitutions~$(iii)$, and they form 
 a hexagon (as in the weak order on $S_3$).
 Below we display the chains, their descent sets, and indicate the
 substitutions~$(iii)$ by two-headed arrows.
\[
  \begin{picture}(366,97)(0,-1)
    \put(0 ,50){$(t_{a\alpha},t_{b\beta},t_{c\gamma})$}
    \put(30,37){\small$\emptyset$}
     
    \put(61,62){\vector(3,2){38}}
    \put(91,82){\vector(-3,-2){30}}
    \put(91,25){\vector(-3,2){30}}
    \put(61,45){\vector(3,-2){38}}

    \put(102,87){$(t_{a\alpha},t_{c\gamma},t_{b\beta})$}
    \put(124,73){\small$\{2\}$}

    \put(102,14){$(t_{b\beta},t_{a\alpha},t_{c\gamma})$}
    \put(124,0){\small$\{1\}$}

    \put(166,90){\vector(1,0){34}}
    \put(196,90){\vector(-1,0){30}}
    \put(196,17){\vector(-1,0){30}}
    \put(166,17){\vector(1,0){34}}

    \put(204,87){$(t_{c\gamma},t_{a\alpha},t_{b\beta})$}
    \put(226,73){\small$\{1\}$}

    \put(204,14){$(t_{b\beta},t_{c\gamma},t_{a\alpha})$}
    \put(226,0){\small$\{2\}$}

    \put(305,62){\vector(-3,2){38}}
    \put(275,82){\vector(3,-2){30}}
    \put(275,25){\vector(3,2){30}}
    \put(305,45){\vector(-3,-2){38}}

    \put(306,50){$(t_{c\gamma},t_{b\beta},t_{a\alpha})$}
    \put(324,37){\small$\{1,2\}$}
 \end{picture}
\]
 As we seek a rule for $\varphi_2$ which uses a minimal number of the substitutions $(i)$,
 $(ii)$, and $(iii)$ from Proposition~\ref{P:substitutions}, we must have that $\varphi_2$
 is the involution given by the two horizontal arrows, which is again 
 included in rule {\bf (B)}.
\end{proof}

 Consequently, there is at most one dual equivalence on chains in the Grassmannian-Bruhat
 order that is local and uniform and given by a minimal number of substitutions, and that
 dual equivalence is reversible. 

\begin{definition}\label{D:wde}
 Suppose that $\zeta$ is a permutation of rank $n$.
 Define involutions $\varphi_i$ for $2\leq i\leq n{-}1$ on chains in $[e,\zeta]_\preceq$
 where $\varphi_i(c)=c$ either if the chain $c$ has a descent in both positions $i{-}1$
 and $i$ or $c$ has no descents in those positions.
 If $c$ has exactly one descent in positions $i{-}1$ and $i$, then $\varphi_i(c)$ is
 obtained from $c$ by applying the involution $\varphi_2$ of Lemma~\ref{L:locality} to the
 transpositions in $c$ at positions $i{-}1$, $i$, and $i{+}1$.
 These involutions $\varphi_i$ act locally and are uniform and reversible, and are the
 unique such involutions for which $c$ and $\varphi_i(c)$ differ by a minimal number of
 substitutions of Proposition~\ref{P:substitutions}.
\end{definition}

\begin{remark}\label{R:Knuth}
 Any interval in Young's lattice is naturally an interval in the Grassmannian-Bruhat order.
 For such an interval, all substitutions have type~$(iii)$ and
 only rule {\bf (B)} applies.
 In this case, it is just the Knuth relation,
 $(\beta,\alpha,\gamma)\leftrightarrow(\beta,\gamma,\alpha)$
 or $(\gamma,\alpha,\beta)\leftrightarrow(\alpha,\gamma,\beta)$, 
 where $\alpha<\beta<\gamma$, which gives Haiman's dual equivalence.

 Whenever rule {\bf (B)} applies, it is simply the Knuth relation on 
 the labels of the chains.
\end{remark}

\begin{example}
 Figure~\ref{F:124536_graph} displays how the involutions $\varphi_2$, $\varphi_3$, and
 $\varphi_4$ act on the nine chains of length five from Figure~\ref{F:124536} in
 Example~\ref{Ex:124536}.
 We indicate the descent set by placing dots at the descents, and indicate which
 involution $\varphi_i$ and rule {\bf (A)}, {\bf (B)}, or {\bf (C)} applies to each edge.
\begin{figure}[htb]
  \begin{picture}(350,115)\thicklines

  \put(  0,100){$t_{36}\dt t_{23}t_{45}\dt t_{34}\dt t_{12}$}
   \put( 70,102){\MyCyan{\line(1,0){65}}}
   \put(100,106){\MyCyan{$\varphi_2$}}
   \put( 98,91){\MyCyan{{\scriptsize\bf (C)}}}
  \put(140,100){$t_{45}t_{36}\dt t_{23}t_{34}\dt t_{12}$}
   \put(210,102){\MyGreen{\line(1,0){65}}}
   \put(240,106){\MyGreen{$\varphi_4$}}
   \put(238, 91){\MyGreen{{\scriptsize\bf (B)}}}
  \put(280,100){$t_{45}t_{36}\dt t_{23}\dt t_{12}t_{34}$}

  \put(  0,50){$t_{36}\dt t_{45}\dt t_{23}t_{34}\dt t_{12}$}
   \put( 33,62){\MyMag{\line(0,1){30}}}
   \put( 18,75){\MyMag{$\varphi_3$}}
   \put( 35,75){\MyMag{{\scriptsize\bf (B)}}}
  \put(140,50){$t_{45}\dt t_{34}t_{46}\dt t_{24}\dt t_{12}$}
   \put(173,62){\MyMag{\line(0,1){30}}}
   \put(176,75){\MyMag{$\varphi_3$}}
   \put(153,75){\MyMag{{\scriptsize\bf (A)}}}
  \put(280,50){$t_{36}\dt t_{23}t_{45}\dt t_{12}t_{34}$}
   \put(313,62){\MyCyan{\line(0,1){30}}}
   \put(316,75){\MyCyan{$\varphi_2$}}
   \put(292,75){\MyCyan{{\scriptsize\bf (C)}}}

  \put(  0,0){$t_{36}\dt t_{45}\dt t_{23}\dt t_{12}t_{34}$}
   \put( 33,12){\MyGreen{\line(0,1){30}}}
   \put( 18,24){\MyGreen{$\varphi_4$}}
   \put( 35,24){\MyGreen{{\scriptsize\bf (B)}}}
  \put(140,0){$t_{35}t_{56}\dt t_{45}\dt t_{24}\dt t_{12}$}
   \put(173,12){\MyCyan{\line(0,1){30}}}
   \put(176,24){\MyCyan{$\varphi_2$}}
   \put(153,24){\MyCyan{{\scriptsize\bf (A)}}}
  \put(280,0){$t_{36}\dt t_{23}\dt t_{12}t_{45}\dt t_{34}$}
   \put(311,12){\MyMag{\line(0,1){30}}}
   \put(296,24){\MyMag{$\varphi_3$}}
   \put(276,24){\MyMag{{\scriptsize\bf (B)}}}
   \put(315,12){\MyGreen{\line(0,1){30}}}
   \put(318,24){\MyGreen{$\varphi_4$}}
   \put(333,24){\MyGreen{{\scriptsize\bf (B)}}}
  \end{picture}
 \caption{The graph from the chains in Figure~\ref{F:124536}.}\label{F:124536_graph}
\end{figure}
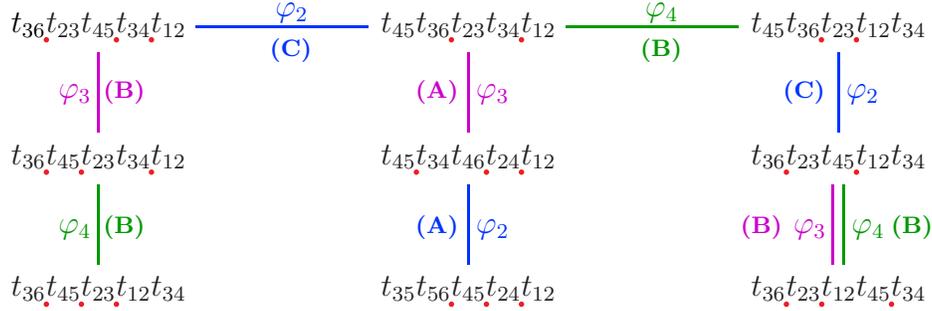
 Below we write the graph, labeling the vertices by the numbers of the
 corresponding chains of Figure~\ref{F:124536} and the edges with $\varphi_i$.
\[
  \begin{picture}(220,49)(-6,1)
   \put(  0, 7){\includegraphics{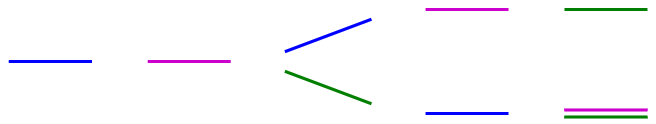}}

   \put( 17,29){\MyCyan{$\varphi_2$}}
   \put( 57,29){\MyMag{$\varphi_3$}}

   \put( 96,38){\MyCyan{$\varphi_2$}}
   \put( 96,10){\MyGreen{$\varphi_4$}}
   \put(137,45){\MyMag{$\varphi_3$}}
   \put(137, 3){\MyCyan{$\varphi_2$}}
   \put(177,45){\MyGreen{$\varphi_4$}}
   \put(177, 2){\MyGreen{$\varphi_4$}}
   \put(177,16){\MyMag{$\varphi_3$}}

   \put( -5.7,21.5){$(1)$}
   \put( 34.6,21.5){$(2)$}
   \put( 74.6,21.5){$(3)$}

   \put(114.7, 6.7){$(5)$}
   \put(114.7,36.7){$(7)$}
   \put(154.9, 6.7){$(8)$}
   \put(154.9,36.7){$(4)$}
   \put(195.4, 6.7){$(9)$}
   \put(195.4,36.7){$(6)$}
  \end{picture}
\]

 For the eleven chains in the interval $[e,(1,4,5,3,2,6)]_\preceq$ of Example~\ref{Ex:143256}
 we have the following graph, where the vertex labels $(1)$---$(11)$ are the same as
 in~\eqref{Eq:SubstitutionGraph}, and the edges are labeled
 with the involution $\varphi_i$ and rule {\bf (A)}, {\bf (B)}, or {\bf (C)} for that edge.
\[
   \begin{picture}(380,112)(-6,-1)
    \put(0,0){\includegraphics{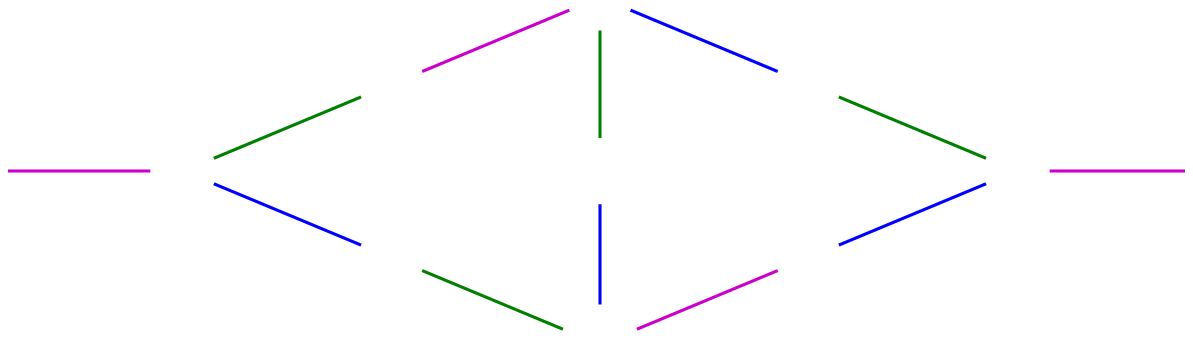}}
    \put( -5.7,51.5){$(6)$}
    \put( 54.6,51.5){$(7)$}
    \put(114.9,76.5){$(1)$}
    \put(114.9,26.5){$(9)$}
    \put(174.9,101.5){$(2)$}
    \put(174.9,51.5){$(8)$}
    \put(172.9, 1.5){$(10)$}
    \put(235.1,76.5){$(3)$}
    \put(236.3,26.5){$(5)$}
    \put(295.4,51.5){$(4)$}
    \put(354.7,51.5){$(11)$}

    \put( 27,60){\MyMag{$\varphi_3$}}
    \put( 25,46){\MyMag{\scriptsize\bf(A)}}
    \put( 84,73){\MyGreen{$\varphi_4$}}
    \put( 91,60){\MyGreen{\scriptsize\bf(A)}}
    \put( 91,45){\MyCyan{$\varphi_2$}}
    \put( 83,32){\MyCyan{\scriptsize\bf(B)}}

    \put(145,98){\MyMag{$\varphi_3$}}
    \put(149,84){\MyMag{\scriptsize\bf(C)}}
    \put(149,22){\MyGreen{$\varphi_4$}}
    \put(143, 8){\MyGreen{\scriptsize\bf(A)}}

    \put(169,75){\MyGreen{$\varphi_4$}}
    \put(184,75){\MyGreen{\scriptsize\bf(B)}}
    \put(169,32){\MyCyan{$\varphi_2$}}
    \put(184,32){\MyCyan{\scriptsize\bf(C)}}

    \put(210,96){\MyCyan{$\varphi_2$}}
    \put(204,83){\MyCyan{\scriptsize\bf(A)}}
    \put(205,22){\MyMag{$\varphi_3$}}
    \put(209, 8){\MyMag{\scriptsize\bf(A)}}

    \put(269,73){\MyGreen{$\varphi_4$}}
    \put(259,60){\MyGreen{\scriptsize\bf(B)}}
    \put(260,45){\MyCyan{$\varphi_2$}}
    \put(269,34){\MyCyan{\scriptsize\bf(C)}}

    \put(326,60){\MyMag{$\varphi_3$}}
    \put(324,46){\MyMag{\scriptsize\bf(B)}}
   \end{picture}
\]
\end{example}

Our main Theorem~\ref{Th:Main} is a consequence of the following result.

\begin{theorem}\label{Th:DualEquivalance}
  The involutions $\varphi_i$ for $2\leq i\leq n{-}1$ of Definition~$\ref{D:wde}$ form a 
  dual equivalence on the set of chains in an interval $[e,\zeta]_\preceq$ of rank $n$
  in the Grassmannian-Bruhat order.
\end{theorem}

\begin{proof}
 Conditions (i), (ii.a), (ii.b), and (iii) of Definition~\ref{D:strong_dual_equivalence} for a
 dual equivalence hold immediately by the definition of $\varphi_i$ and
 Lemma~\ref{L:locality}. 

 For Condition (ii.c), suppose that exactly one of the chains $c$ and $\varphi_i(c)$ has a
 descent at position $i{-}2$.
 Then $c\neq\varphi_i(c)$ and exactly one of $c$ and $\varphi_i(c)$ has a descent at
 position $i{-}1$, by Condition (ii.b).
 Without any loss, assume that $c$ has a descent at position $i{-}1$ and thus it has no
 descent at position $i$.
 Then, in the positions $i{-}1,i,i{+}1$, the chain $c$ is one of the triples in the first
 column in the rules {\bf (A)}, {\bf (B)}, {\bf (C)} used to define of $\varphi_i$.
 Since exactly one of $c$ and $\varphi_i(c)$ has a descent at position $i{-}2$, that one
 must be $\varphi_i(c)$, as the label on the transposition of $c$ in position $i{-}1$ is
 at least as large as the corresponding label for $\varphi_i(c)$.
 It follows that in positions $i{-}2$ and $i{-}1$, $c$ has a descent only at $i{-}1$
 and $\varphi_i(c)$ has a descent only at $i{-}2$.
 Therefore, by Property (i), $\varphi_{i-1}(c)\neq c$ and
 $\varphi_{i-1}(\varphi_i(c))\neq\varphi_i(c)$, which shows (ii.c).
 The Condition (ii.d) follows by reversing the chains and using the reversibility of the
 involutions $\varphi_i$.

 Consider now the remaining conditions (iv.a), (iv.b), and (iv.c) of    
 Definition~\ref{D:dual_equivalence}.
 These conditions may be verified by checking all subchains of all chains, where the
 subchains have lengths up to five for (iv.a) and (iv.b), and those of length six for
 (iv.c).
 By Proposition~\ref{P:properties}(1), this is equivalent to checking all chains in
 intervals $[e,\zeta]_\preceq$ in the Grassmannian-Bruhat order of ranks up to five and
 six, as the involutions $\varphi_i$ are local.
 As the involutions are uniform, this is a finite set of intervals.
 Indeed, if $\zeta$ is a permutation of rank six, then $\zeta$ is the product of six
 transpositions, and these transpositions involve at most twelve different numbers.
 Thus $\zeta=\iota_I(\eta)$ for some permutation $\eta$ in $S_{12}$.

 Thus we may complete the proof by generating all chains in the Grassmannian-Bruhat order
 of lengths up to six, up to the equivalence $\zeta$ is equivalent to $\iota_I(\eta)$.
 We have done just that and have written software that generates the
 chains and verifies Conditions (iv.a), (iv.b), and (iv.c).
 There are 1236 equivalence classes of chains of length four, 29400 equivalence classes of
 chains of length five, and 881934 equivalence classes of chains of length six.
 Given this set of equivalence classes of chains of length $n$ for $n=4,5,6$, the software
 determines the involutions $\varphi_i$ for $i=2,\dotsc,n{-}1$.
 From these data, the software determines all connected colored graphs $\calG$ whose
 vertices are this set of chains where $c$ and $\varphi_i(c)$ are connected by an edge of
 color $i$ if $c\neq\varphi_i(c)$.
 Then these graphs are partitioned into isomorphism classes, where the isomorphism
 respects the edge labels and descent sets of the vertices.
 Finally, the conditions (iv.a), (iv.b), and (iv.c) are checked on representatives of these
 isomorphism classes.
 Thus these graphs are all dual equivalence graphs, which completes the proof.
\end{proof}

 The software for these tasks is available on the web\fauxfootnote{{\tt
     http://www.math.tamu.edu/\~{}sottile/research/pages/positivity/}} , as well as
 documentation, the sets of chains, connected dual equivalence graphs, and isomorphism
 class representatives. 
 Some of this verification may be done by hand or inspection and it is not necessary to
 generate all chains of length six to verify (iv.c).
 We discuss this more in detail, including giving all dual equivalence graphs of chains of
 lengths four and five, as well as one for chains of length six in the subsections that
 follow. 
 The numbers of chains, graphs, and classes of graphs is displayed in Table~\ref{T:classes}.
\begin{table}[htb]
 \caption{Numbers of chains and graphs}\label{T:classes}
 \begin{tabular}{|l||r|r|r|r||}\hline
   &\multicolumn{4}{|c|}{Numbers}\\\hline
   $n$&Chains&dual equivalence graphs&Isomorphism Classes& $\{\omega,\rho\}$-classes\\\hline\hline
    4 &  1236&  499&  7& 4\\\hline
    5 & 29500& 5948& 28&12\\\hline
    6 &881934&82294&178&73\\\hline
%
%
 \end{tabular}
\end{table}

\subsection{Chains of pairwise disjoint transpositions}
 As we noted in Remark~\ref{disjoint}, a valid chain of length $n$ in the Grassmannian
 Bruhat order consisting of pairwise disjoint transpositions is equivalent to one in
 $S_{2n}$.
 These are in turn given by one of $n!$ orderings of the pairs in a complete non-crossing
 matching on $[2n]$, and the number of such matchings is the Catalan number $C_n$.
 These are in bijection with parenthesizations of n objects, and thus with the vertices of
 the associahedron and with plane binary trees.
 Thus up to equivalence there are $n! C_n$ chains in the Grassmannian Bruhat order
 consisting of pairwise disjoint transpositions.

 As observed in Remark~\ref{R:Knuth}, the involutions on such a chain (which come from
 rule {\bf (B)}) amount to applications of the Knuth relations on the labels of a chain.
 Consequently, the colored graph $\calG$ constructed from an interval $[e,\zeta]_\preceq$
 where $\zeta\in S_{2n}$ has rank $n$ and one chain (hence all chains) is composed of
 pairwise disjoint transpositions, is isomorphic to the colored graph $\calG_n$
 constructed from the symmetric group $S_n$ coming from Knuth equivalences.
 Since $\calG_n$ is a strong dual equivalence (each component corresponds to all the 
 Young tableaux of a given shape under Haiman's dual equivalence), we conclude that
 $\calG$ is a strong dual equivalence graph.

 Thus in the proof of Theorem~\ref{Th:DualEquivalance} we did not need a
 computer to verify the conditions for these chains of disjoint transpositions.
 Table~\ref{T:disjoint} shows the number of chains of rank $n$ for $n=3,4,5,6$ together
 with the numbers that are composed of pairwise disjoint transpositions.
\begin{table}[htb]
 \caption{Numbers of disjoint chains}\label{T:disjoint}
  \begin{tabular}{|c||r|r|r|r|}\hline
   $n$&3&4&5&6\\\hline
   $n!C_n$&30&336&5040&95040\\\hline
   Number of Chains&70&1236&29500&881934\\\hline
  \end{tabular}
\end{table}

By Proposition~\ref{P:properties} there is a bijection $c\leftrightarrow c^\vee$, given by
reversal, between chains in $[e,\zeta]_\preceq$ and chains in $[e,\zeta^{-1}]_\preceq$.
Since $\Des(c^\vee)=\rho\omega\Des(c)$, for every graph $\calG$ coming from a set of
chains in the Grassmannian-Bruhat order, we have also the graph $\rho\omega\calG$ coming
from reversal of the chains giving $\calG$.
Thus we will always obtain both quasisymmetric functions $K_\calG$ and 
$K_{\rho\omega\calG}=\omega K_\calG$.

\subsection{Dual equivalence graphs from chains of length four}\label{SS:four}
There are four $\{\rho,\omega\}$-classes of dual equivalence graphs coming from chains of length 4.
\[
  \begin{picture}(8,20)
   \put(0,7){\includegraphics{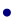}}
  \end{picture}
 \qquad
  \begin{picture}(36,20)
   \put( 0, 7){\includegraphics{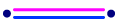}}
   \put(15, 0){\MyCyan{\scriptsize$2$}}
   \put(15,13){\MyMag{\scriptsize$3$}}
  \end{picture} 
    \qquad
   \begin{picture}(64,20)
   \put( 0, 7){\includegraphics{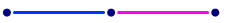}}
   \put(15,12){\MyCyan{\scriptsize$2$}}
   \put(45,12){\MyMag{\scriptsize$3$}}
  \end{picture} 
    \qquad
  \begin{picture}(124,20)
   \put(  0,7){\includegraphics{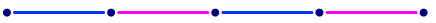}}
   \put( 15,12){\MyCyan{\scriptsize$2$}}
   \put( 45,12){\MyMag{\scriptsize$3$}}
   \put( 75,12){\MyCyan{\scriptsize$2$}}
   \put(105,12){\MyMag{\scriptsize$3$}}
  \end{picture}
\]
The first three are strong dual equivalence graphs, and these four give the symmetric
functions
\[
   s_{(4)}\,,\ s_{(2,2)}\,,\ s_{(3,1)}\,,\ \mbox{ and }\ s_{(2,2)}+s_{(3,1)}\,,
\]
respectively, where we write only one of $K_{\calG}$ and 
$K_{\rho\omega\calG}=\omega K_{\calG}$ when the two functions differ.

\subsection{Dual equivalence graphs from chains of length five}\label{SS:five}
There are twelve $\{\rho,\omega\}$-classes of dual equivalence graphs coming from chains of length
five, which we give in Figure~\ref{F:GFive}.
\begin{figure}[htb]
\[
  \begin{picture}(8,25)(0,-5)
   \put( 0,5){\includegraphics{1.eps}}
  \end{picture}
  \quad
  \begin{picture}(92,25)(0,-5)
   \put( 0,5){\includegraphics{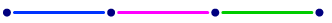}}
   \put(14,10){\MyCyan{\scriptsize$2$}}
   \put(44,10){\MyMag{\scriptsize$3$}}
   \put(74,10){\MyGreen{\scriptsize$4$}}
  \end{picture} 
    \quad
  \begin{picture}(122,22)(0,-2)
   \put(0,7){\includegraphics{g5_5.eps}}

   \put( 14, 1){\MyCyan{\scriptsize$2$}}
   \put( 14,13){\MyMag{\scriptsize$3$}}
   \put( 44,12){\MyGreen{\scriptsize$4$}}
   \put( 74,12){\MyCyan{\scriptsize$2$}}
   \put(104,13){\MyGreen{\scriptsize$4$}}
   \put(104, 1){\MyMag{\scriptsize$3$}}

  \end{picture}
  \quad
  \begin{picture}(124,26)
   \put(0,2){\includegraphics{g5_6.eps}}

   \put( 14,15){\MyMag{\scriptsize$3$}}
   \put( 42,19){\MyGreen{\scriptsize$4$}}
   \put( 42, 1){\MyCyan{\scriptsize$2$}}
   \put( 75,19){\MyCyan{\scriptsize$2$}}
   \put( 75, 1){\MyGreen{\scriptsize$4$}}
   \put(104,15){\MyMag{\scriptsize$3$}}
  \end{picture}\vspace{10pt}
\]
\[
  \begin{picture}(154,42)
   \put(  0, 6){\includegraphics{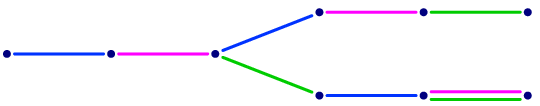}}
   \put( 14,23){\MyCyan{\scriptsize$2$}}
   \put( 44,23){\MyMag{\scriptsize$3$}}
   \put( 73,28){\MyCyan{\scriptsize$2$}}
   \put( 73, 7){\MyGreen{\scriptsize$4$}}
   \put(104, 1){\MyCyan{\scriptsize$2$}}
   \put(104,35){\MyMag{\scriptsize$3$}}

   \put(134, 0){\MyGreen{\scriptsize$4$}}
   \put(134,13){\MyMag{\scriptsize$3$}}
   \put(134,35){\MyGreen{\scriptsize$4$}}
  \end{picture}
   \qquad
  \begin{picture}(215,36)
   \put(  0,4.5){\includegraphics{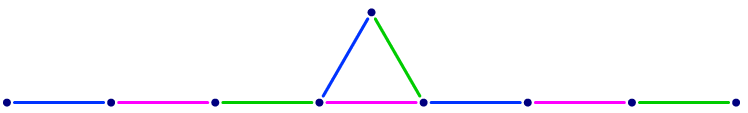}}
   \put( 14,0){\MyCyan{\scriptsize$2$}}
   \put( 44,0){\MyMag{\scriptsize$3$}}
   \put( 74,0){\MyGreen{\scriptsize$4$}}
   \put(104,0){\MyMag{\scriptsize$3$}}
   \put( 92,18){\MyCyan{\scriptsize$2$}}
   \put(117,18){\MyGreen{\scriptsize$4$}}
   \put(134,0){\MyCyan{\scriptsize$2$}}
   \put(164,0){\MyMag{\scriptsize$3$}}
   \put(194,0){\MyGreen{\scriptsize$4$}}
  \end{picture}\vspace{10pt}
\]
\[
  \begin{picture}(185,63)
   \put(  0, 7){\includegraphics{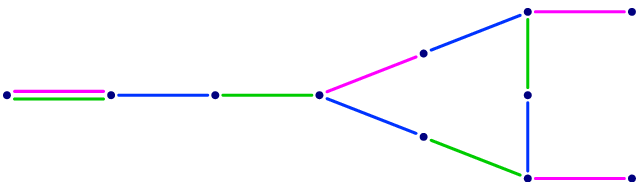}}
   \put( 14,35){\MyMag{\scriptsize$3$}}
   \put( 14,24){\MyGreen{\scriptsize$4$}}
   \put( 44,34){\MyCyan{\scriptsize$2$}}
   \put( 74,34){\MyGreen{\scriptsize$4$}}
   \put(103,40){\MyMag{\scriptsize$3$}}
   \put(103,19){\MyCyan{\scriptsize$2$}}
   \put(133,52){\MyCyan{\scriptsize$2$}}
   \put(133, 7){\MyGreen{\scriptsize$4$}}

   \put(154,40){\MyGreen{\scriptsize$4$}}
   \put(154,18){\MyCyan{\scriptsize$2$}}
   \put(164, 0){\MyMag{\scriptsize$3$}}
   \put(164,58){\MyMag{\scriptsize$3$}}
  \end{picture}
    \qquad
  \begin{picture}(185,57)
   \put( 0,  0){\includegraphics{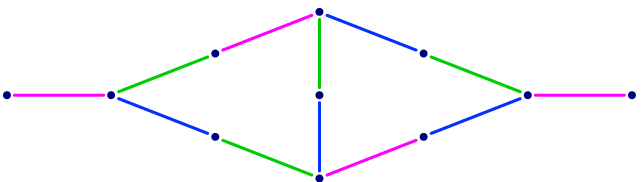}}
   \put( 14,27){\MyMag{\scriptsize$3$}}
   \put( 43,36){\MyGreen{\scriptsize$4$}}
   \put( 43,12){\MyCyan{\scriptsize$2$}}
   \put( 73,45){\MyMag{\scriptsize$3$}}
   \put( 73, 0){\MyGreen{\scriptsize$4$}}
   \put( 94,33){\MyGreen{\scriptsize$4$}}
   \put( 94,11){\MyCyan{\scriptsize$2$}}

   \put(106,45){\MyCyan{\scriptsize$2$}}
   \put(106, 0){\MyMag{\scriptsize$3$}}
   \put(136,36){\MyGreen{\scriptsize$4$}}
   \put(136,12){\MyCyan{\scriptsize$2$}}
   \put(164,27){\MyMag{\scriptsize$3$}}
  \end{picture}\vspace{10pt}
\]
\[
  \begin{picture}(223,33)
   \put(  0,0.5){\includegraphics{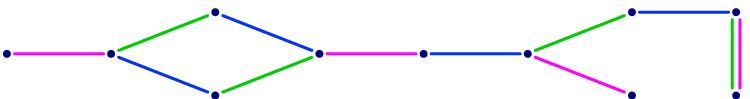}}
   \put( 14,15){\MyMag{\scriptsize$3$}}
   \put( 43,22){\MyGreen{\scriptsize$4$}}
   \put( 43, 0){\MyCyan{\scriptsize$2$}}
   \put( 75,22){\MyCyan{\scriptsize$2$}}
   \put( 75, 0){\MyGreen{\scriptsize$4$}}
   \put(104,15){\MyMag{\scriptsize$3$}}
   \put(134,15){\MyCyan{\scriptsize$2$}}
   \put(163,22){\MyGreen{\scriptsize$4$}}
   \put(163, 0){\MyMag{\scriptsize$3$}}
   \put(194,27){\MyCyan{\scriptsize$2$}}

   \put(206,10){\MyGreen{\scriptsize$4$}}
   \put(216,10){\MyMag{\scriptsize$3$}}
  \end{picture}\vspace{10pt}
\]
\[
  \begin{picture}(220,64)
   \put(0,7){\includegraphics{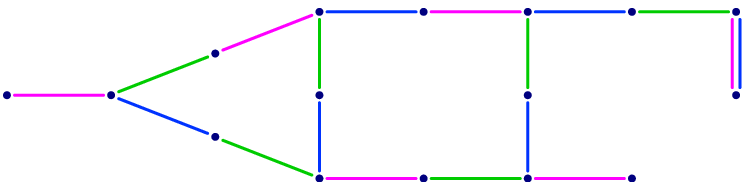}}
   \put( 14,34){\MyMag{\scriptsize$3$}}
   \put( 44,41){\MyGreen{\scriptsize$4$}}
   \put( 44,19){\MyCyan{\scriptsize$2$}}
   \put( 74,53){\MyMag{\scriptsize$3$}}
   \put( 74, 7){\MyGreen{\scriptsize$4$}}
   \put( 94,42){\MyGreen{\scriptsize$4$}}
   \put( 94,18){\MyCyan{\scriptsize$2$}}

   \put(104,58){\MyCyan{\scriptsize$2$}}
   \put(104, 0){\MyMag{\scriptsize$3$}}

   \put(134,58){\MyMag{\scriptsize$3$}}
   \put(134, 0){\MyGreen{\scriptsize$4$}}
   \put(154,42){\MyGreen{\scriptsize$4$}}
   \put(154,18){\MyCyan{\scriptsize$2$}}

   \put(164,58){\MyCyan{\scriptsize$2$}}
   \put(164, 0){\MyMag{\scriptsize$3$}}
   \put(194,58){\MyGreen{\scriptsize$4$}}

   \put(216,42){\MyCyan{\scriptsize$2$}}
   \put(205,42){\MyMag{\scriptsize$3$}}
  \end{picture}
    \qquad
  \begin{picture}(196,64)
   \put(6,7){\includegraphics{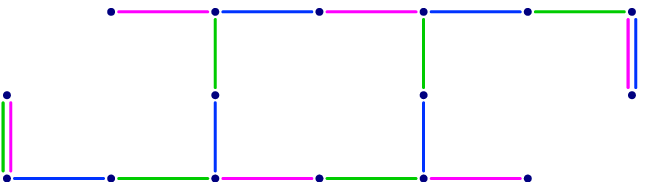}}
   \put( 11,18){\MyMag{\scriptsize$3$}}
   \put(  0,18){\MyGreen{\scriptsize$4$}}
   \put( 20, 0){\MyCyan{\scriptsize$2$}}
   \put( 50,58){\MyMag{\scriptsize$3$}}
   \put( 50, 0){\MyGreen{\scriptsize$4$}}

   \put( 70,42){\MyGreen{\scriptsize$4$}}
   \put( 70,18){\MyCyan{\scriptsize$2$}}

   \put( 80,58){\MyCyan{\scriptsize$2$}}
   \put( 80, 0){\MyMag{\scriptsize$3$}}
   \put(110,58){\MyMag{\scriptsize$3$}}
   \put(110, 0){\MyGreen{\scriptsize$4$}}

   \put(130,42){\MyGreen{\scriptsize$4$}}
   \put(130,18){\MyCyan{\scriptsize$2$}}

   \put(140,58){\MyCyan{\scriptsize$2$}}
   \put(140, 0){\MyMag{\scriptsize$3$}}

   \put(170,58){\MyGreen{\scriptsize$4$}}

   \put(192,42){\MyCyan{\scriptsize$2$}}
   \put(181,42){\MyMag{\scriptsize$3$}}
  \end{picture}\vspace{10pt}
\]
\[
  \begin{picture}(251,64)
   \put(6,7){\includegraphics{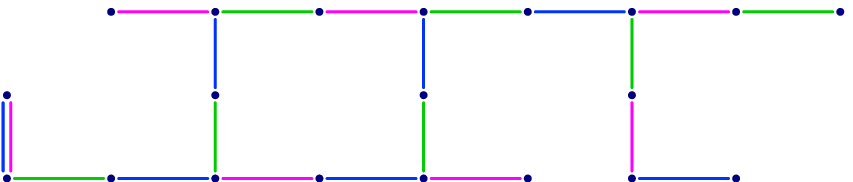}}
   \put( 11,18){\MyMag{\scriptsize$3$}}
   \put(  0,18){\MyCyan{\scriptsize$2$}}
   \put( 20, 0){\MyGreen{\scriptsize$4$}}
   \put( 50,58){\MyMag{\scriptsize$3$}}
   \put( 50, 0){\MyCyan{\scriptsize$2$}}

   \put( 70,42){\MyCyan{\scriptsize$2$}}
   \put( 70,18){\MyGreen{\scriptsize$4$}}

   \put( 80,58){\MyGreen{\scriptsize$4$}}
   \put( 80, 0){\MyMag{\scriptsize$3$}}
   \put(110,58){\MyMag{\scriptsize$3$}}
   \put(110, 0){\MyCyan{\scriptsize$2$}}

   \put(130,42){\MyCyan{\scriptsize$2$}}
   \put(130,18){\MyGreen{\scriptsize$4$}}

   \put(140,58){\MyGreen{\scriptsize$4$}}
   \put(140, 0){\MyMag{\scriptsize$3$}}

   \put(170,58){\MyCyan{\scriptsize$2$}}

   \put(190,42){\MyGreen{\scriptsize$4$}}
   \put(190,18){\MyMag{\scriptsize$3$}}

   \put(200,58){\MyMag{\scriptsize$3$}}
   \put(200, 0){\MyCyan{\scriptsize$2$}}

   \put(230,58){\MyGreen{\scriptsize$4$}}

  \end{picture}
\]
\caption{Dual equivalence graphs in the Grassmannian Bruhat order for $n=5$}\label{F:GFive}
\end{figure}
The first four are all strong dual equivalence graphs.
For each of these dual equivalence graphs $\calG$, the pair of symmetric functions
$K_{\calG}$ and $K_{\omega\rho\calG}=\omega K_{\calG}$ depend only upon the number of
nodes and are given in Table~\ref{T:GFive}, 
\begin{table}[htb]
 \caption{Symmetric functions of graphs vs. number of vertices}\label{T:GFive}
 \begin{tabular}{|l||c|c|c|c|c|c||}\hline
 Vertices&1&4&5&6&9&11\\\hline
 $K_\calG$&$s_{(5)}$&$s_{(4,1)}$&$s_{(3,2)}$&$s_{(3,1,1)}$&$s_{(4,1)}+s_{(3,2)}$
         &$s_{(3,2)}+s_{(3,1,1)}$\raisebox{-5pt}{\rule{0pt}{3pt}}\\\hline
 \end{tabular}\vspace{5pt}

 \begin{tabular}{|l||c|c||}\hline
  Vertices&16&20\\\hline\raisebox{-5pt}{\rule{0pt}{3pt}}
  $K_\calG$&$s_{(2,2,1)}+s_{(3,1,1)}+s_{(3,2)}$&$s_{(2,2,1)}+s_{(3,1,1)}+s_{(3,2)}+s_{(4,1)}$\\\hline
 \end{tabular}
\end{table}
where we write only one of the two functions $K_\calG$ and $K_{\rho\omega\calG}$.
This census for $n=4$ and $n=5$ demonstrates Condition (iv.a) of local Schur positivity.

\subsection{Dual equivalence graphs from chains of length six}\label{SS:six}

Table~\ref{T:six} gives the numbers of graphs from chains of length six by the numbers of
vertices, as well as the numbers of isomorphism classes and the different symmetric functions.
\begin{table}[htb]
 \caption{Graphs and symmetric functions for $n$=6}\label{T:six}

 \begin{tabular}{|c||c|c|c|c|c|c|c|c|c|c||}\hline
  vertices&1&5&9&10&14&16&19&21&26&35 \\\hline\hline
  graphs&1806&20922&18594&19828&796&16134&786&414&948&738\\\hline
  $\{\rho,\omega\}$-classes&1&2&1&1&4&1&7&3&8&6 \\\hline
  Symm.\  Fns.\ &2&4&2&2&4&1&4&2&2&2 \\\hline
 \end{tabular}\vspace{5pt}

 \begin{tabular}{|c||c|c|c|c|c|c|c|c|c|c|c|c|c|c||}\hline
  vertices&37&40&42&45&47&54&56&59&61&66&75&80&91&96 \\\hline\hline
  graphs&14&408&48&22&254&53&370&12&54&28&10&48&4&4\\\hline
  $\{\rho,\omega\}$-classes&1&10&1&2&3&2&6&1&6&3&1&1&1&1 \\\hline
  Symm.\ Fns.\ &2&4&2&4&2&3&4&2&6&2&2&1&2&1 \\\hline
 \end{tabular}
\end{table}

\subsection{Condition (iv.b)}
Condition (iv.b) may be verified by inspection of the dual equivalence graphs when
$n=5$. 
For this, we have $i=3$, and we first need an edge with label 3 (so that
$c\neq\varphi_i(c)$) for which both endpoints admit one edge with label 2 and one edge with
label 4. 
Among the 12 dual equivalence graphs, this only occurs for the middle edge in the graph
$\calG$ on nine vertices containing a triangle.
We display this graph, labeling the vertices with their descent sets.
\[
  \begin{picture}(292,56)(-5,-3)
   \put(  0,4.5){\includegraphics{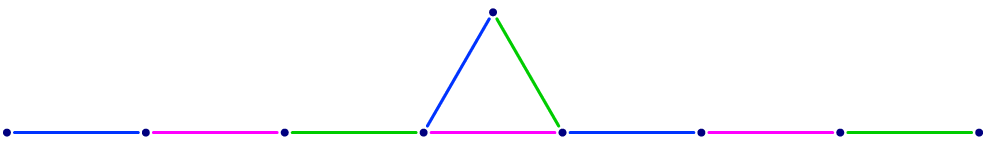}}
   \put( -5,-2){\scriptsize$\{1\}$}   \put( 19,10){\MyCyan{\scriptsize$2$}}
   \put( 35,-2){\scriptsize$\{2\}$}   \put( 59,10){\MyMag{\scriptsize$3$}}
   \put( 75,-2){\scriptsize$\{3\}$}   \put( 99,10){\MyGreen{\scriptsize$4$}}
   \put(112,-2){\scriptsize$\{2,4\}$} \put(139,10){\MyMag{\scriptsize$3$}}
   \put(125,23){\MyCyan{\scriptsize$2$}}
   \put(132,47){\scriptsize$\{1,4\}$}
   \put(155,23){\MyGreen{\scriptsize$4$}}
   \put(152,-2){\scriptsize$\{1,3\}$}\put(179,10){\MyCyan{\scriptsize$2$}}
   \put(195,-2){\scriptsize$\{2\}$} \put(219,10){\MyMag{\scriptsize$3$}}
   \put(235,-2){\scriptsize$\{3\}$} \put(259,10){\MyGreen{\scriptsize$4$}}
   \put(275,-2){\scriptsize$\{4\}$} 
  \end{picture}
\]
Note that $\rho\calG=\calG$.
Condition (iv.b) considers the connected subgraph containing this middle edge and
any edges with labels 2 and 3, as well as that containing this middle edge and any edges
with labels 3 and 4.
Each consists of five vertices, and for the first, we restrict each descent set to $\{1,2,3\}$
and for the second we restrict each descent set to $\{2,3,4\}$, and subtract one.
Condition (iv.b) is that the quasisymmetric functions of each of these restricted
subgraphs are equal.
Here is the subgraph with edge labels 2 and 3, and the restricted graph.
\[
  \begin{picture}(140,56)(110,-3)
   \put(120,4.5){\includegraphics{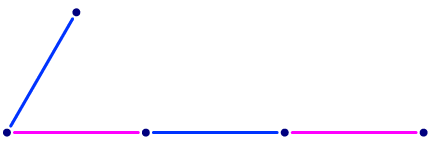}}
   \put(110,-2){\scriptsize$\{2,4\}$} \put(139,10){\MyMag{\scriptsize$3$}}
   \put(125,23){\MyCyan{\scriptsize$2$}}
   \put(132,47){\scriptsize$\{1,4\}$}
   \put(152,-2){\scriptsize$\{1,3\}$}\put(179,10){\MyCyan{\scriptsize$2$}}
   \put(195,-2){\scriptsize$\{2\}$} \put(219,10){\MyMag{\scriptsize$3$}}
   \put(235,-2){\scriptsize$\{3\}$}
  \end{picture}
  \qquad
  \begin{picture}(175,56)(75,-3)
   \put( 80,4.5){\includegraphics{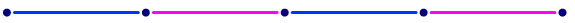}}
   \put( 75,-2){\scriptsize$\{1\}$} \put( 99,10){\MyCyan{\scriptsize$2$}}
   \put(115,-2){\scriptsize$\{2\}$} \put(139,10){\MyMag{\scriptsize$3$}}
   \put(152,-2){\scriptsize$\{1,3\}$}\put(179,10){\MyCyan{\scriptsize$2$}}
   \put(195,-2){\scriptsize$\{2\}$} \put(219,10){\MyMag{\scriptsize$3$}}
   \put(235,-2){\scriptsize$\{3\}$}
  \end{picture}
\]
Here is the subgraph with edge labels 3 and 4, and the restricted graph.
\[
  \begin{picture}(140,56)(35,-3)
   \put( 40,4.5){\includegraphics{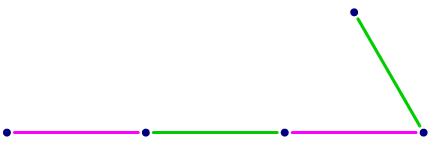}}
   \put( 35,-2){\scriptsize$\{2\}$}   \put( 59,10){\MyMag{\scriptsize$3$}}
   \put( 75,-2){\scriptsize$\{3\}$}   \put( 99,10){\MyGreen{\scriptsize$4$}}
   \put(112,-2){\scriptsize$\{2,4\}$} \put(139,10){\MyMag{\scriptsize$3$}}
   \put(132,47){\scriptsize$\{1,4\}$}
   \put(155,23){\MyGreen{\scriptsize$4$}}
   \put(152,-2){\scriptsize$\{1,3\}$}
  \end{picture}
  \qquad
  \begin{picture}(180,56)(75,-3)
   \put( 80,4.5){\includegraphics{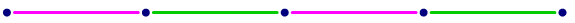}}
   \put( 75,-2){\scriptsize$\{2\}$}  \put( 99,10){\MyMag{\scriptsize$3$}}
   \put(115,-2){\scriptsize$\{3\}$}  \put(139,10){\MyGreen{\scriptsize$4$}}
   \put(152,-2){\scriptsize$\{2,4\}$}\put(179,10){\MyMag{\scriptsize$3$}}
   \put(195,-2){\scriptsize$\{3\}$}  \put(219,10){\MyGreen{\scriptsize$4$}}
   \put(235,-2){\scriptsize$\{4\}$} 

  \end{picture}
\]
Observe that if we subtract 1 from all descents and edge labels in the second restricted
graph, we obtain the first restricted graph.
This shows that Condition (iv.b) holds.
We note that we could also deduce this fact for $\calG$ using that $\rho\calG=\calG$.

\subsection{Condition (iv.c)}\label{SS:iv.c}
By Proposition~\ref{P:properties}, the reversal of a chain in the interval
$[e,\zeta]_\preceq$ gives a chain in $[e,\zeta^{-1}]_\preceq$ whose edge labels are the
reverse of the original chain. 
Thus if $\calG$ is the colored graph constructed from labeled chains in
$[e,\zeta]_\preceq$, then $\rho\omega\calG$ is the colored graph constructed from labeled
chains in $[e,\zeta^{-1}]_\preceq$.
Thus $\calG$ satisfies both conditions (iv.c) (the one given in detail {\it and} its
reversal) if and only if both $\calG$ and $\rho\omega\calG$ satisfy the same one of these 
two conditions.
Thus we will only check the condition given in detail, for all graphs from chains of
length $n=6$.

This condition concerns 5-edges incident to flat 4-chains in a graph $\calG$.
A flat 4-chain in is a list $c_1,c_2,\dotsc,c_{2r}$ of distinct vertices of
$\calG$ where, for each $j=1,\dotsc,r$ the vertices $c_{2j-1}$ and $c_{2j}$ are connected
by a 4-edge.
Furthermore, for each $j=1,\dotsc,r{-}1$, the vertex $c_{2j}$ is connected to $c_{2j+1}$ by an
alternating sequence of 2- and 3- edges, beginning and ending with a 2-edge,
and $c_{2j+1}$ is {\it not} incident to a 3-edge.
From our classification of graphs for chains of length 4 in Subsection~\ref{SS:four}, these
connections are one of two types. 
We display them below, showing the 2- and 3- edges between $c_{2j}$ and
$c_{2j+1}$, as well as those incident on $c_{2j}$ and $c_{2j-1}$.
\[
   \begin{picture}(84,20)
   \put( 0, 7){\includegraphics{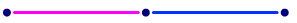}}
   \put(20,12){\MyMag{\scriptsize$3$}} \put(36,2){\scriptsize$c_{2j}$}
   \put(60,12){\MyCyan{\scriptsize$2$}}\put(72,2){\scriptsize$c_{2j+1}$}
  \end{picture} 
    \qquad
  \begin{picture}(164,20)
   \put(  0,7){\includegraphics{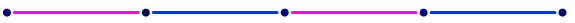}}
   \put( 20,12){\MyMag{\scriptsize$3$}}  \put( 36,2){\scriptsize$c_{2j}$}
   \put( 60,12){\MyCyan{\scriptsize$2$}}
   \put(100,12){\MyMag{\scriptsize$3$}}
   \put(140,12){\MyCyan{\scriptsize$2$}}  \put(152,2){\scriptsize$c_{2j+1}$}
  \end{picture}
\]
 As $c_{2j+1}$ is not incident to any 3-edge, we must have that $c_{2j}$ is
 incident to a 3-edge.

 Then Condition (iv.c) asserts that for every flat 4-chain $c_1,\dotsc,c_{2r}$, if we have 
 that both $c_{2j-1}$ and $c_{2j}$ admit a 5-edge for some  $1<j<r$, then either each of
 the first $2j$ vertices $c_1,\dotsc,c_{2j-1},c_{2j}$ admit a 5-edge, or else 
 each of the last $2r{-}2j{+}2$ vertices $c_{2j-1},c_{2j},\dotsc,c_{2r}$ admit a 5-edge.
 This condition is vacuous except for $r\geq 3$.

 Figure~\ref{F:21} shows a dual equivalence graph $\calG$ with a flat 4-chain.
 This graph with 21 vertices comes from the chain
 $(t_{27},t_{45},t_{12},t_{34},t_{56},t_{45})$ in the interval 
\begin{figure}[htb]
  \begin{picture}(255,172)(5,0)
   \put(10,0){\includegraphics{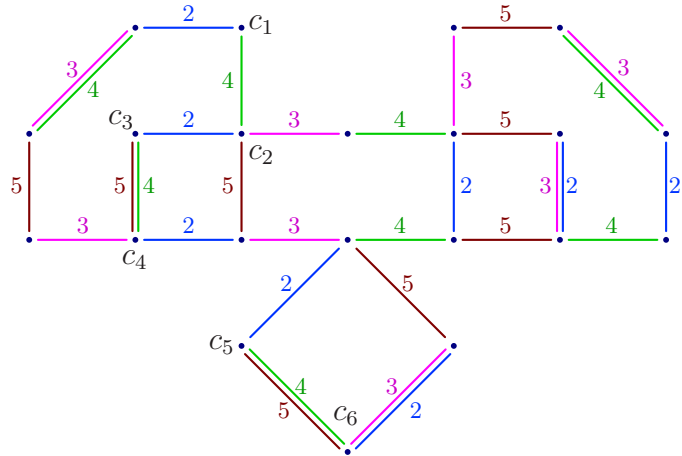}}

   \put(  5,100){\MyMar{\scriptsize$5$}}
   \put( 30, 85){\MyMag{\scriptsize$3$}}
   \put( 26,144){\MyMag{\scriptsize$3$}}
   \put( 34,136.5){\MyGreen{\scriptsize$4$}}
   \put( 44,100){\MyMar{\scriptsize$5$}}  \put(42,126){$c_3$}
   \put( 55,100){\MyGreen{\scriptsize$4$}}
   \put( 70, 85){\MyCyan{\scriptsize$2$}} \put(47, 73){$c_4$}
   \put( 70,125){\MyCyan{\scriptsize$2$}}
   \put( 70,165){\MyCyan{\scriptsize$2$}}
   \put( 85,140){\MyGreen{\scriptsize$4$}} \put(95,162){$c_1$}
   \put( 85,100){\MyMar{\scriptsize$5$}}   \put(95,113){$c_2$}
                                           \put(80, 41){$c_5$}
   \put(110, 85){\MyMag{\scriptsize$3$}}
   \put(110,125){\MyMag{\scriptsize$3$}}
   \put(107, 63){\MyCyan{\scriptsize$2$}}
   \put(106, 15){\MyMar{\scriptsize$5$}}
   \put(112.5, 24){\MyGreen{\scriptsize$4$}}
                                            \put(127,14){$c_6$}
   \put(150,125){\MyGreen{\scriptsize$4$}}
   \put(150, 85){\MyGreen{\scriptsize$4$}}
   \put(153, 63){\MyMar{\scriptsize$5$}}
   \put(147, 24){\MyMag{\scriptsize$3$}}
   \put(156, 15){\MyCyan{\scriptsize$2$}}

   \put(175,140){\MyMag{\scriptsize$3$}}
   \put(175,100){\MyCyan{\scriptsize$2$}}
   \put(190, 85){\MyMar{\scriptsize$5$}}
   \put(190,125){\MyMar{\scriptsize$5$}}
   \put(190,165){\MyMar{\scriptsize$5$}}

   \put(205,100){\MyMag{\scriptsize$3$}}
   \put(215,100){\MyCyan{\scriptsize$2$}}
   \put(230, 85){\MyGreen{\scriptsize$4$}}

   \put(234.5,144.5){\MyMag{\scriptsize$3$}}
   \put(225.5,136){\MyGreen{\scriptsize$4$}}

   \put(254,100){\MyCyan{\scriptsize$2$}}

  \end{picture}
\caption{A dual equivalence graph with a flat 4-chain}\label{F:21}
\end{figure}
 $[e,(1,2,7)(3,5,4,6)]_\preceq$, which corresponds to the lower left vertex of $\calG$.
 Since flat 4-chains involve only 2-, 3-, and 4-edges, and $\varphi_i(c)$ has the same
 sixth transposition as $c$ for $i=2,3,4$, a flat 4-chain in a graph with $n=6$ comes by
 appending the same sixth transposition to all chains of vertices in a flat 4-chain in a
 graph with $n=5$.
 For example, the flat 4-chain in the graph of Figure~\ref{F:21} comes from the flat
 4-chain in the graph with $n=5$ depicted below.
\[
  \begin{picture}(196,66)(0,-1)
   \put(6,7){\includegraphics{g5_16a.eps}}
   \put( 11,18){\MyMag{\scriptsize$3$}}
   \put(  0,18){\MyGreen{\scriptsize$4$}}
   \put( 20, 0){\MyCyan{\scriptsize$2$}} 
   \put( 50,58){\MyMag{\scriptsize$3$}}
   \put( 50, 0){\MyGreen{\scriptsize$4$}}
     \put(32,-1){$c_1$} \put(62,-1){$c_2$} \put(55,30){$c_3$}
     \put(62,62){$c_4$} \put(152,62){$c_5$} \put(185,62){$c_6$}
   \put( 70,42){\MyGreen{\scriptsize$4$}}
   \put( 70,18){\MyCyan{\scriptsize$2$}}

   \put( 80,58){\MyCyan{\scriptsize$2$}}
   \put( 80, 0){\MyMag{\scriptsize$3$}}
   \put(110,58){\MyMag{\scriptsize$3$}}
   \put(110, 0){\MyGreen{\scriptsize$4$}}

   \put(130,42){\MyGreen{\scriptsize$4$}}
   \put(130,18){\MyCyan{\scriptsize$2$}}

   \put(140,58){\MyCyan{\scriptsize$2$}}
   \put(140, 0){\MyMag{\scriptsize$3$}}

   \put(170,58){\MyGreen{\scriptsize$4$}}

   \put(192,42){\MyCyan{\scriptsize$2$}}
   \put(181,42){\MyMag{\scriptsize$3$}}
  \end{picture}
\]

 Thus it is possible to check Condition (iv.c) by first generating all flat 4-chains in
 graphs for $n=5$, and then considering all possible ways to append a transposition to all 
 the chains in the Grassmannian-Bruhat order appearing in those flat 4-chains.

\providecommand{\bysame}{\leavevmode\hbox to3em{\hrulefill}\thinspace}
\providecommand{\MR}{\relax\ifhmode\unskip\space\fi MR }
\providecommand{\MRhref}[2]{%
  \href{http://www.ams.org/mathscinet-getitem?mr=#1}{#2}
}
\providecommand{\href}[2]{#2}

\end{document}